\documentclass[11pt]{amsart}
\usepackage{hyperref}
\usepackage{geometry}                % See geometry.pdf to learn the layout options. There are lots.
\usepackage{amssymb}
\usepackage{enumitem}
\usepackage{amsmath}
\usepackage{amsthm}
\usepackage{mathrsfs}
\usepackage{mathtools} 		 
\usepackage{color}

\theoremstyle{plain}
{
  \newtheorem{theo}{Theorem}[section]
  \newtheorem{lem}[theo]{Lemma}
  \newtheorem{cor}[theo]{Corollary}
  \newtheorem{prop}[theo]{Proposition}
  \newtheorem*{CEP}{The Capelli Eigenvalue Problem}
   \newtheorem*{rCEP}{The Refined Capelli Eigenvalue Problem}
\newtheorem*{Ack}{Acknowledgements}  

\newtheorem*{Not}{Notation and Conventions}  

}
  
\newenvironment{Acknowledgement}{\begin{Ack}\rmfamily\upshape}{\end{Ack}}

\newenvironment{Notation}{\begin{Not}\rmfamily\upshape}{\end{Not}}

\theoremstyle{definition}
{
   \newtheorem{defn}[theo]{Definition}
   \newtheorem{eg}[theo]{Example}
}

\theoremstyle{remark}
{
  \newtheorem{remark}[theo]{Remark}
}

\newcommand{\abs}[1]{\vert #1 \vert}  %I don't have the package physics.sty
\newcommand{\C}{\mathbb{C}} 
\newcommand{\gl}{\mathfrak{\lowercase{gl}}} 
\newcommand{\mat}[1]{\left[\begin{matrix} #1 \end{matrix}\right]}

\newcommand{\p}{\mathscr{P}}
\newcommand{\s}{\mathscr{S}}
\newcommand{\pd}{\mathscr{PD}}

\newcommand{\h}{\mathfrak{\lowercase{h}}} 
\newcommand{\br}{\mathfrak{\lowercase{b}}} 

\newcommand{\g}{\mathfrak{\lowercase{g}}}

\newcommand{\Z}{\mathbb{Z}}

\newcommand{\pr}[1]{\left(#1\right)}
\newcommand{\lr}[1]{\langle#1\rangle}

\newcommand{\SPo}{P_{\mu,1}}

\newcommand{\weird}{\Gamma}
\newcommand\definecommand[2]{%%
  \expandafter\def\csname#1l\endcsname##1{\underline{#2_{##1}}}%%
}

\definecommand{u}{\lambda}
\begin{document}
 
\title[The Refined Solution to the CEP]{The Refined Solution to the Capelli Eigenvalue Problem for $\gl(m|n)\oplus\gl(m|n)$ and $\gl(m|2n)$}
\author{Mengyuan Cao}\email{mcao027@uottawa.ca}
\author{Monica Nevins}\email{mnevins@uottawa.ca}
\author{Hadi Salmasian}\email{hsalmasi@uottawa.ca}
\address{Department of Mathematics and Statistics, University of Ottawa, STEM Complex,
150 Louis-Pasteur Pvt,
Ottawa, Ontario
Canada K1N 6N5 Canada}
\date{\today}                                           % Activate to display a given date or no date

\maketitle
\begin{abstract}
{Let $\g$ be either the Lie superalgebra  $\gl(V)\oplus\gl(V)$
  where $V:=\C^{m|n}$ 
or the Lie superalgebra $\gl(V)$
where $V:=\C^{m|2n}$. Furthermore, let $W$ be the $\g$-module defined by  
 $W:=V\otimes V^*$ in the former case and  $W:=\s^2(V)$ in the latter case.
Associated to $(\g,W)$ there exists a distinguished basis 
of \emph{Capelli operators} 
$\left\{D^\lambda\right\}_{\lambda\in\Omega}$, naturally indexed by a set of hook partitions $\Omega$, for the  subalgebra of 
$\g$-invariants in
the superalgebra $\pd(W)$  
of superdifferential operators on $W$.

Let $\mathfrak b$ be a Borel subalgebra of $\g$. 
We compute eigenvalues of the $D^\lambda$ on the irreducible $\g$-submodules of $\p(W)$ and obtain them explicitly as the  evaluation of the interpolation super Jack polynomials of Sergeev--Veselov at suitable affine functions of the $\mathfrak b$-highest weight.  While the former case is straightforward, the latter is significantly more complex.  This generalizes a result by Sahi, Salmasian and Serganova for these cases, where such formulas were given for a fixed choice of Borel subalgebra.}

%When $(\g,\mathfrak k)=(\gl(V)\oplus\gl(V),\gl(V))$ the eigenvalue of $D^\lambda$ is obtained by evaluation of the  interpolation super Jack polynomials of Sergeev-Veselov at the $\mathfrak b$-highest weight. However, when $(\g,\mathfrak k)=(\gl(V),\mathfrak{osp}(V))$ the situation is substantially more complex. %In this case, first we demonstrate by a counterexample that one cannot expect the eigenvalues of the $D^\lambda$ to be expressible as evaluation of  a single interpolation super Jack polynomial on the highest weight. Nevertheless, 
%We prove that the eigenvalue is a polynomial function of $\tau_\br(\underline{\lambda}_\br)$, where $\tau_\br$ is a \emph{piecewise} affine map   on the span of $\br$-highest weights of the irreducible submodules of $\p(W)$.
\end{abstract}

\section{Introduction}
Throughout this paper the base field will be $\C$. 
Let $\g$ be a Lie superalgebra and let $W$ be a $\g$-module. Then the superalgebra $\p(W)$ of superpolynomials on $W$ has a canonical $\g$-module structure, and  $\p(W)\cong\s(W^*)$ as $\g $-modules. 
Suppose that $\p(W)$ has a multiplicity-free decomposition into irreducible $\g$-modules, that is   
\begin{equation}\label{cap_3}\p(W)\cong \bigoplus_{\lambda\in \Omega}W_{\lambda},\end{equation} where the $W_\lambda$'s are pairwise non-isomorphic irreducible $\g$-modules 
and $\Omega$ is an index set. In addition, assume that 
the $W_\lambda$ are modules of type M, that is,  
$\mathrm{End}_\g(W_\lambda)\cong \C$. 

Let $\mathscr D(W)$ denote the algebra of constant-coefficient superdifferential operators on $W$. Then there is  a canonical isomorphism
$\mathscr D(W)\cong \s(W)$.  Hence, as explained in Appendix~\ref{appxA}, we have a $\mathfrak g$-module decomposition 
 \[
\mathscr D(W)\cong \bigoplus_{\lambda\in\Omega}W_\lambda^*.
\]
Let $\mathscr{PD}(W)$ denote the algebra of superdifferential operators on $W$ with superpolynomial coefficients. We have an isomorphism
of $\g $-modules \[
\pd(W)\cong \p(W)\otimes \mathscr D(W).
\]
Taking $\g$-invariant elements of both sides, it follows that  we have vector-space isomorphisms
\begin{equation}\label{eq:pdg_mul_free_decom}\pd(W)^\g\cong 
\bigoplus_{\lambda,\mu\in\Omega} (W_\lambda\otimes W_\mu^*)^\g
\cong\bigoplus_{\lambda\in\Omega} \mathbb{C}id_{W_{\lambda}}.
\end{equation}
%We are interested in the (super)algebra $\pd(W)^\g$, the algebra of $\g$-invariant polynomial-coefficient differential operators. % Let $\mathcal{A}$ be the span of the set of highest weights $\{\underline{\lambda}\mid \lambda\in\Omega\}$ where $\underline{\lambda}$ denotes the highest weight parameterized by $\lambda\in\Omega$.  

%\begin{lem}\label{lem:pdg_mul_free}As a $\g$-module, we have that \begin{equation}\label{eq:pdg_mul_free_decom}\pd(W)^\g\cong \bigoplus_{\lambda\in\Omega} \mathbb{C}id_{W_{\lambda}}\end{equation}
%\end{lem}
%\begin{proof}
%Let $\s(W)$ be the super-commutative algebra of $W$. First recall that $\pd(W)\cong\p(W)\otimes\s(W)$ as $\g$-modules. By Equation \eqref{cap_3} and the fact that $\p(W)\cong\s(W^*)$ as $\g$-modules, we have that \begin{align*}
%\pd(W)^\g\cong \left(\p(W)\otimes \mathscr{S}(W)\right)^\g&\cong \left(\p(W)\otimes \p(W^*)\right)^\g\\
%&\cong \left(\bigoplus_{\lambda\in\Omega}W_\lambda\otimes \bigoplus_{\mu\in\Omega }W_\mu^*\right)^\g\\
%&\cong \bigoplus_{\lambda,\mu\in\Omega}\left(W_\lambda\otimes W_\mu^*\right)^\g\\
%&\cong  \bigoplus_{\lambda,\mu\in\Omega}\mathrm{Hom}_\g\left(W_\mu, W_\lambda\right)\\
%&\cong\bigoplus_{\lambda\in\Omega}\mathrm{Hom}_\g(W_\lambda,W_\lambda)\\
%&\cong \bigoplus_{\lambda\in\Omega} \mathbb{C}id_{W_{\lambda}}
%\end{align*}where the last two isomorphisms follow from Schur's lemma \cite[Lemma 3.4]{cheng_wang_2013}.  \end{proof}

\begin{defn}\label{defn:capelli_operator}
The Capelli operator $D^\lambda$ for $\lambda\in \Omega$ is the $\g$-invariant differential operator in $\pd(W)^\g$ that corresponds to $id_{W_\lambda}\in\mathrm{Hom}_\g(W_\lambda,W_\lambda)$ via the isomorphism \eqref{eq:pdg_mul_free_decom}. 
\end{defn}
Note that $id_{W_\lambda}$ corresponds to the tensor $\sum_{i=1}^{d_\lambda}w_i^*\otimes w_i$, where 
$d_\lambda:=\dim W_\lambda$ and the sets of vectors
$\{w_i\}_{i=1}^{d_\lambda}$ 
and $\{w_i^*\}_{i=1}^{d_\lambda}$
are bases of $W_\lambda^*$ and $W_\lambda$ that are dual to each other.
The above choices give a well-defined normalization of the operator $D^\lambda$. Indeed as we explain in Appendix~\ref{appxA}, if $W_\lambda\subseteq \p^d(W)$ then
$D^\lambda$ acts on $W_\lambda$ by the scalar $d!$.

Clearly the set $\{D^\lambda\mid\lambda\in \Omega\}$ forms a basis for $\pd(W)^\g.$
By Schur's Lemma, the operators $D^\mu$ for $\mu\in\Omega$ act on the  irreducible components $W_\lambda$ of $\p(W)$ by scalars $c_\mu(\lambda)$. Thus we have the  following natural problem.

\begin{CEP} Find the eigenvalue $c_\mu(\lambda)$ for each $\lambda,\mu\in\Omega$. \end{CEP}

The Capelli Eigenvalue Problem (henceforth CEP) goes back to the work of Kostant and Sahi \cite{kostant_sahi_1991,kostant_sahi_1993}. It has been studied extensively and spawned the theory of interpolation Jack polynomials \cite{sahi_1994,okounkov_olshanski_1997}.
If $\g $ is the complexification of $\mathfrak {gl}_n(\mathbb F)$ where $\mathbb F=\mathbb R,\C,\mathbb H$ and $W$ is the complexification of the space of hermitian $n\times n$ matrices with entries in $\mathbb F$, then it follows from the work of Knop and Sahi~\cite{Knop1996DifferenceEA} that 
the
eigenvalue $c_\mu(\lambda)$ is given by a polynomial function of the form $P_\mu(\lambda+\rho)$, where $\rho$ is the usual shift by half the sum of positive roots, and $P_\mu$ is a symmetric polynomial whose top-degree homogeneous component yields a distinguished class of special functions that occur in spherical analysis on Riemannian symmetric spaces, known as  Jack polynomials (at parameter values $\frac{1}{2},1,2$ respectively). Interpolation Jack polynomials also occur as eigenfunctions of the Calogero-Moser-Sutherland operator, which describes the hamiltonian of the quantum $n$-body problem in one dimension~\cite{Sergeev2005GeneralisedDD}.

In this paper we are interested in a refinement of the CEP when $(\g,W)$ is either the pair $(\gl(V)\oplus\gl(V),V\otimes V^*)$ with $V:=\C^{m|n}$ or the pair $(\gl(V),\s^2(V))$ with $V:=\C^{m|2n}$. The CEP for these pairs (and several others that are obtained by the TKK construction) was studied 
in~\cite{sahi_salmasian_serganova_2019}. 
%(For the connection between the CEP and the symmetric pair $(\g,\mathfrak k)$, as well as the explanation for the special value of the parameter $\theta$ \textcolor{red}{used in this paper} in terms of deformed root systems,  see~\cite{sahi_salmasian_serganova_2019}.)
As explained in \emph{loc. cit.}, these two cases  are related to the diagonal symmetric pair  $(\gl(V)\oplus\gl(V),\gl(V))$  and the  symmetric pair $(\gl(V),\mathfrak{osp}(V))$, respectively.  
%(For further details, as well as the explanation for the special value of the parameter $\theta$ used here in terms of deformed root systems,  see~\cite{sahi_salmasian_serganova_2019}.)
From this viewpoint, the present paper serves as a starting point to  extend the results of~\cite{sahi_salmasian_serganova_2019} to arbitrary Borel subalgebras. The complex picture of families of Borel subalgebras that arises for the pair $(\gl(V),\mathfrak{osp}(V))$
suggests that a complete and uniform general solution to the refined CEF requires new ideas.

In~\cite{sahi_salmasian_serganova_2019} the authors first parametrize the representations of $\g$ by hook partitions. Then they fix a Borel subalgebra $\br$ (coming from the standard Borel subalgebra or its opposite; see \cite[Table 4]{sahi_salmasian_serganova_2019}), and write down a formula for the $\br$-highest
weights of irreducible components of $\p(V)$ in terms of the hook partition. Finally they show that a formula for the eigenvalue $c_\mu(\lambda)$ can be computed as an interpolation super Jack polynomial evaluated at an affine function of 
%a polynomial in 
the $\br$-highest weight.  %because "interpolation super Jack" comes up in a couple of paragraphs

The fixed choice of the Borel subalgebra  essentially means that the formula for
the eigenvalues of a Capelli operator is really dependent on the parametrization
of modules by partitions. 
Our goal in this paper is to remove this constraint. That is we compute formulas for $c_\mu(\lambda)$ which are polynomial in the $\br$-highest weights of the $W_\lambda$, for \emph{arbitrary} Borel subalgebras $\br$. 
In other words, our goal in this paper is to study  the following variation of the CEP for the aforementioned pairs $(\g,W)$.

\begin{rCEP}
For any Borel subalgebra $\br$, find the eigenvalue $c_\mu(\lambda)$ as a polynomial function in the $\br$-highest weight of $W_\lambda$. 
\end{rCEP}

Since the Borel subalgebras $\gl(m|n)$ are not conjugate by the inner automorphism group, the solution to the CEP does \emph{not} carry over to the refined form. Indeed, the work of Section~\ref{sec:cpeglm2n} shows that the choice of affine function varies significantly with the type of Borel subalgebra.
%Indeed in Section~\ref{ss:surprisingexample-5} we demonstrate through an explicit example that it is not possible to obtain a solution to the refined CEP by a single polynomial.  However, we show that one can explicitly partition $\h^*$, and hence the set of $\br$-highest weights,  into finitely many parts such that on each part the eigenvalue is obtained by evaluation of an interpolation super Jack polynomial on a simple affine function $\tau_\br$ evaluated on the highest weight $\underline\lambda_\br$ of $W_\lambda$. 

\medskip

Here is an outline of the paper. In Section~\ref{sec:Interpolation} we recall the definitions and properties of the interpolation super Jack polynomials of Sergeev and Veselov~\cite{Sergeev2005GeneralisedDD}. In Section~\ref{sec:CEPdiag} we show that for the diagonal symmetric pair 
$(\gl(m|n)\oplus\gl(m|n), \C^{m|n}\otimes (\C^{m|n})^*)$ the solution of the refined CEP is obtained in a relatively straightforward way from  the results of \cite{sahi_salmasian_serganova_2019}; this result is stated in Theorem~\ref{T:1}.  

We begin the study of the pair $(\gl(m|2n), \mathcal{S}^2(\C^{m|2n}))$ in Section~\ref{section:bsam2n}, where we establish formulae for $\br$-highest weights, for any decreasing Borel subalgebra $\br$, and derive necessary conditions for a solution to the refined CEP.  
In Section~\ref{sec:cpeglm2n}, we solve the refined CEP in a series of results of increasing complexity: for \emph{very even} Borel subalgebras in Section~\ref{SS:veryeven}; for \emph{generic} highest weights in Section~\ref{SS:generic}; for \emph{relatively even} Borel subalgebras in Section~\ref{SS:relativelyeven} and finally in all cases in Section~\ref{ss:full}.  % evenTheorem~\ref{T:veryeven} for \emph{very even} Borel subalgebras; Corollary~\ref{prop:reiscomp}  \emph{relatively even} Borel subalgebras; and finally Theorem~\ref{theo:lastone} that encompasses all Borel subalgebras.  
%Finally, in Section~\ref{S:polyinterp}, we explore the optimality of our result with, first, an example demonstrating the impossibility of a simple solution in general, and second, a more complex polynomial interpolation in the general case. 
In Appendix~\ref{appxA}, we carefully prove the well-definedness of the CEP and establish the normalization of the Capelli operators.

\begin{Acknowledgement}
Part of this work appeared in the first author's PhD thesis \cite{MengyuanPhD}.  The second author's research is  partially funded by NSERC Discovery Grant RGPIN-2020-05020.  The third author's research is partially funded by NSERC Discovery Grant RGPIN-2018-04044.
\end{Acknowledgement}

\begin{Notation} 
Since all Cartan subalgebras of $\gl(V)$ are conjugate by inner automorphisms, for the refined CEP it suffices to consider all Borel subalgebras that contain a fixed Cartan subalgebra.
Throughout this paper $\h_{m|n}\cong \C^{m|n}$ denotes the standard (\emph{i.e.}, diagonal) Cartan subalgebra of $\gl(m|n)$; we denote the standard basis of $\h_{m|n}$, or equivalently of $\C^{m|n}$, indiscriminately by $\{e_i\mid 1\leq i\leq m+n\}$. We also denote the  dual basis  for $\h_{m|n}^*$ by  $\{\epsilon_i, \delta_j\mid 1\leq i \leq m, 1\leq j \leq n\}$. Henceforth $(\cdot,\cdot)$ is the nondegenerate supersymmetric form on $\h_{m|n}^*$ that is diagonal with respect to this basis and satisfies $(\epsilon_i,\epsilon_i)=-(\delta_j,\delta_j)=1$ for all $1\leq i \leq m, 1\leq j \leq n$.
Let $\mathfrak{b}_{m|n}$ denote the standard upper triangular Borel subalgebra of $\gl(m|n)$ and $\mathfrak{b}^{\mathrm{op}}_{m|n}$ its opposite.  The set of Borel subalgebras has several orbits under the Weyl group.  A representative of each orbit (called an \emph{increasing} Borel subalgebra) may be obtained from $\br_{m|n}$ by performing a sequence of simple odd reflections; such a Borel subalgebra again contains $\h_{m|n}$.  Alternatively, applying simple odd reflections to $\mathfrak{b}^{\mathrm{op}}_{m|n}$ produces another set of representatives, whose elements are the \emph{decreasing} Borel subalgebras.

If $\lambda = (\lambda_1,\lambda_2,\ldots)$ is an integer partition, then we denote by $\lambda' = (\lambda_1', \lambda_2', \ldots)$ the transpose partition $\lambda_i' = |\{j\mid \lambda_j\geq i\}|$.
Let $\mathscr{H}(m|n)$ denote the set of $m|n$ hook partitions, \emph{i.e.}, integer partitions $(\lambda_1,\lambda_2,\ldots)$ satisfying $\lambda_{m+1}\leq n$. Doubling the elements of such a partition yields an element of $\mathscr{H}_2(m|2n)$, which is the subset of $m|2n$ hook partitions whose parts are all even.
 Each $\lambda \in \mathscr{H}(m|n)$  corresponds to a unique irreducible finite-dimensional representation of $\gl(m|n)$, denoted $V_{\lambda}$, whose highest weight with respect to the standard Borel subalgebra $\br_{m|n}$ is given by $\sum_i \lambda_i\epsilon_i+\sum_j\max\{\lambda'_{j}-m, 0\}\delta_j$.  Note that its negative is then the highest weight  of $V_\lambda^*$ with respect to $\br_{m|n}^{\mathrm{op}}$.

\end{Notation}
%\begin{Acknowledgement}
% The research of H.S. was partially supported by
%an NSERC Discovery Grant (RGPIN-2018-04044).
%The research of M.N. was partially supported by 
%an NSERC Discovery Grant (RGPIN-2020-05020).
%\end{Acknowledgement}

\section{Interpolation super Jack Polynomials}

\label{sec:Interpolation}
%Our solution of the refined CEP relies on the interpolation super Jack polynomials of Sergeev and Veselov, which we define below.
Let $m,n\in\mathbb{Z}_{\ge0}$ and let $R_{m|n}:=\mathbb{C}\left[x_1,\ldots,x_m, y_1\ldots,y_n\right]$ be the ring of polynomials in $m+n$ indeterminates with coefficients in $\mathbb{C}$. Let $\theta\in\C\setminus \mathbb Q^{\leq 0}$.  We say a polynomial $f\in R_{m|n}$ is \emph{separately symmetric} if $f$ is symmetric in $\{x_i\}_{1\le i\le m}$ and in $\{y_j\}_{1\le j\le n}$ separately.  
%We first define an important class of separately symmetric polynomials. %Then we give the definition of interpolation super Jack polynomial. For more detail, see \cite{Sergeev2005GeneralisedDD}.  	

\begin{defn}We denote  by $\Lambda_{m,n,\theta}$ the subalgebra of separately symmetric polynomials $f$ such that 
\begin{equation}\label{eq:mosys}
f\pr{\ldots,x_i+1/2,\ldots, \ldots,y_j-1/2,\ldots}=f\pr{\ldots,x_i-1/2,\ldots, \ldots,y_j+1/2,\ldots}
\end{equation}on every hyperplane $x_i+\theta y_j=0$ for all $1\le i\le m$ and $1\le j\le n.$ We call the symmetry property defined in Equation \eqref{eq:mosys} the \emph{monoidal symmetry} of $f$.
\end{defn}

\begin{defn}Let $u,v\in\C^{m+n}$. We say $u$ is equivalent to $v$ if $f(u)=f(v)$ for all $f\in\Lambda_{m,n,\theta}$, in which case we write $u\sim_E v$.   We say that $u$ and $v$ are \emph{monoidally equivalent} if they are related by a sequence of separate symmetries and monoidal symmetries, and write $u\sim v$ in this case.  % if $u$ is equivalent to $v$. 
\end{defn}

Evidently if $u\sim v$ then $u\sim_E v$.  As we discuss in more detail (and use) in Section~\ref{ss:full}, these two notions of equivalence differ when there exist elements in $\C^{m|n}$ whose orbits under $\sim$ are infinite.

\begin{lem}\label{lem:equforf}
Let $u=(x_1,\ldots,x_m,y_1,\ldots,y_n)\in\C^{m+n}$ and suppose $1\leq i\leq m$, $1\leq j \leq n$.  Then we have
\[
u\sim \begin{cases}
u-e_i+e_{m+j} & \text{whenever $x_i+\theta y_j = \frac12(1-\theta)$, and}\\
u+e_i-e_{m+j} & \text{whenever $x_i+\theta y_j = -\frac12(1-\theta)$.}
\end{cases}
\]
%\[u\sim u-e_i+e_{m+j}\]
%whenever $x_i+\theta y_j = \frac12(1-\theta)$, and
%\[u\sim u+e_i-e_{m+j}\]whenever $x_i+\theta y_j = -\frac12(1-\theta)$.%The polynomials $f(x|y)$ in $ \Lambda_{m,n,\theta}$ satisfy that whenever
%$x_i+\theta y_j = \frac12(1-\theta)$ we have
%$$
%f(x|y) = f(x-e_i|y+f_j)
%$$
%whereas when $x_i+\theta y_j = -\frac12(1-\theta)$, we have
%$$
%f(x|y) = f(x+e_i|y+f_j).
%$$
\end{lem}

\begin{proof}
For any $v=(x_1', \cdots, x_m',y_1', \cdots, y_n')\in \C^{m+n}$, we have  $v + \frac12 e_i-\frac12 e_{m+j}\sim v-\frac12e_i+\frac12 e_{m+j}$ whenever $x_i'+\theta y_j'=0$.  If $u=v+\frac12 e_i-\frac12 e_{m+j}$ then $x_i'+\theta y_j'=0$ if and only if $x_i+\theta y_j = (x_i'+\frac12)+\theta(y_j'-\frac12)=\frac12(1-\theta)$, yielding the first equivalence.  Alternately, if $u= v-\frac12 e_i+\frac12 e_{m+j}$, then $x_i'+\theta y_j'=0$ if and only if $x_i+\theta y_j = (x_i'-\frac12)+\theta(y_j'+\frac12) = -\frac12(1-\theta)$, and the second equivalence holds.
\end{proof} 
We need one extra definition before we can define the  interpolation  super Jack polynomials. Given an $(m,n)$-hook partition $\lambda\in\mathscr{H}(m,n)$, the \emph{twisted Frobenius coordinates} $\pr{p(\lambda),q(\lambda)}\coloneqq\pr{p_1(\lambda),\ldots,p_m(\lambda),q_1(\lambda),\ldots,q_n(\lambda)}$ of $\lambda$ are defined by
\begin{equation}\label{eq:tw_F_c}\begin{split}p_i(\lambda)\coloneqq\lambda_i-\theta\pr{i-\frac12}-\frac12\pr{n-\theta m}\mbox{, and}\\q_j(\lambda)\coloneqq\lr{\lambda_j'-m}-\theta^{-1}\pr{j-\frac12}+\frac12\pr{\theta^{-1}n+m},\end{split}\end{equation}where $1\le i\le m, 1\le j\le n$ and $\lr{x}\coloneqq\max\{0,x\}$ for all $x\in\mathbb{R}$.  
%Set $E_{i,\theta}=-\theta\pr{i-\frac12}-\frac12\pr{n-\theta m}$ and $F_{j,\theta}=-\theta^{-1}\pr{j-\frac12}+\frac12\pr{\theta^{-1}n+m}$. If $\theta$ is clear from the context, we may write $E_{i,\theta}=E_i$ and $F_{j,\theta}=F_{j}$ for convenience. Thus we have
%\[p_i(\lambda)=\lambda_i+E_{i,\theta} \mbox{ and }q_j(\lambda)=\lr{\lambda_j'-m}+F_{j,\theta}.\]The expressions $E_{i,\theta}$ and $F_{j,\theta}$ will appear frequently when we discuss the Capelli Eigenvalue Problem in later chapters.

\begin{defn}(\cite[Section 6]{Sergeev2005GeneralisedDD})
\label{defn:InterSuperJackPoly}
Let $\lambda\in\mathscr{H}(m,n)$ and let $\theta\in \C\setminus \mathbb Q^{\leq 0}$.  %Let $\pr{p(\mu),q(\mu)}$ be the twisted Frobenius coordinate defined in Equation \eqref{eq:tw_F_c}. 
The interpolation super Jack polynomial 
$P_{\lambda,\theta} $ is the  polynomial in 
$R_{m|n}$ that is uniquely  determined by the following properties:
\begin{enumerate}[label=(\roman*)]
\item${P}_{\lambda,\theta}\in\Lambda_{m,n,\theta}$;
\item deg$({P}_{\lambda,\theta})\le \abs{\lambda}$ where the degree of ${P}_{\lambda,\theta}$ means the total degree in both $x$ and $y$;
\item ${P}_{\lambda,\theta}(p(\lambda),q(\lambda))=d!$ where $d:=|\lambda|$;
 %is given in Equation \eqref{eq:SJP_value}; 

\item ${P}_{\lambda,\theta}(p(\mu),q(\mu))=0$ for all $(m,n)$-hook partitions $\mu$ such that $\abs{\mu}\le\abs{\lambda}$ and $\mu\neq\lambda.$
\end{enumerate}
\end{defn}

\begin{remark}
%The polynomials $P_{\lambda,\theta}$ were originally introduced in~\cite{Sergeev2005GeneralisedDD} but with a different normalization, denoted there by $SP^*_{\lambda}(x,y,\theta)$.   More precisely, in\cite{Sergeev2005GeneralisedDD} it is assumed that the value of ${SP}^*_{\lambda}(p(\lambda),q(\lambda),\theta)$ is given by a hook-type product. 
More precisely, the polynomials $SP^*_{\lambda}(x,y,\theta)$ introduced in~\cite{Sergeev2005GeneralisedDD} differ from the $P_{\lambda,\theta}$ by a normalization:  for each $\lambda$, the value of ${SP}^*_{\lambda}(p(\lambda),q(\lambda),\theta)$ is instead given by a hook-type product. 
 \end{remark}

\begin{remark}\label{rmk:ISJP} 
The set $\{{P}_{\lambda,\theta}\}_{\lambda\in\mathscr{H}(m,n)}$  forms a basis for $\Lambda_{m,n,\theta}$.
Indeed one can prove that there exists a polynomial $\mathbf{P}_{\lambda}$ in $m+n$ variables with coefficients in the field of rational functions in a parameter $\bar\theta$ such that the specialization of $\mathbf{P}_\lambda$ at $\bar\theta:=\theta$ yields $P_{\lambda,\theta}$. In particular, the coefficients of $\mathbf{P}_\lambda$ do not have any poles in $\C\setminus \mathbb Q^{\leq 0}$.  For these facts and further properties of the polynomials $\mathbf P_\lambda$ see~~\cite{Sergeev2005GeneralisedDD}.
\end{remark}

\section{The refined CEP for $\pr{\gl(V)\oplus\gl(V),V\otimes V^*}$}
\label{sec:CEPdiag}

In this section we set $\theta:=1$, $V:=\C^{m|n}$, $\g\coloneqq \gl(m|n)\oplus\gl(m|n)$, $W:=V\otimes V^*$, and $\h\coloneqq \h_{m|n}\oplus \h_{m|n}$.   Then $\br_0\coloneqq \br^{\mathrm{op}}_{m|n}\oplus\br_{m|n}$ is a Borel subalgebra of $\g$ containing $\h$.  
The polynomial algebra $\mathscr{P}(W)$ is a completely reducible and multiplicity-free $\g$-module, with decomposition
\[
\mathscr{P}(W)=\bigoplus_{\lambda\in\mathscr{H}(m|n)}V_{\lambda}^*\otimes V_\lambda^{}.
\]
Let us denote  
the highest weight of $V_{\lambda}$ with respect to $\br_{m|n}$ by $\underline{\lambda}_0$; then  the $\br_{m|n}^{\mathrm{op}}$-highest weight of  $V_{\lambda}^*$ is $-\underline{\lambda}_0$.

We begin by   restating the solution to the Capelli Eigenvalue Problem from \cite{sahi_salmasian_serganova_2019} in terms of the Weyl vector of $\br_0$. 
Write $\tilde{x} \in \C^{m|n}$ for the image of $x \in \h^*_{m|n}$ under the standard isomorphism that sends $\epsilon_i$ to $e_i$ and $\delta_j$ to $e_{m+j}$.  
Let $\tau_0 \colon \h^* \to \C^{m|n}$ denote the affine map 
$$
\tau_0(x,y) =\tilde{y}+\tilde{\rho_0}
$$
where we identify $(x,y)\in \h^*$ with a pair of vectors $x,y\in \C^{m|n}$ and $\rho_0$ denotes the Weyl vector corresponding to $\br_{m|n}$.  Observing that
\[
\rho_0 = \sum_{i=1}^m\frac{m-n+1-2i}{2}\epsilon_i+\sum_{j=1}^n\frac{m+n+1-2j}{2}\delta_{j},
\]
we can rephrase the result of  \cite[Table 3]{sahi_salmasian_serganova_2019} pertaining to $\g$ as follows.

\begin{theo}\cite[Theorem 1.13]{sahi_salmasian_serganova_2019}\label{hadi_theo}
Let $\g$ and $W$ be as above. 
For every $\lambda,\mu \in\mathscr{H}(m|n)$,  the eigenvalue of the Capelli operator $D^{\mu}$  on the irreducible summand $V_{\lambda}^*\otimes V_{\lambda}^{}$ with $\br_0$-highest weight $ (-\underline{\lambda}_{0},\underline{\lambda}_{0})$ is equal to  
\[
\SPo\circ\tau_0 \pr{-\underline{\lambda}_0,\underline{\lambda}_0},
\] 
where $\SPo\in\Lambda_{m,n,1}$ is the interpolation super Jack polynomial associated to $\mu$ for parameter value $\theta=1$.
\end{theo}

% previously said both were the highest weights of $V_\lambda$
Now suppose $\br = \br_1\oplus \br_2$ is an arbitrary Borel subalgebra of $\g$ containing $\h$.  Then $\br_i$ is a Borel subalgebra of $\gl(m|n)$ for $i=1,2$ and we denote its Weyl vector $\rho_{\br_i}$.  Write $\underline{\lambda}_{\br_1}$ for the highest weight of $V_\lambda^*$ with respect to $\br_1$ and $\underline{\lambda}_{\br_2}$ for the highest weight of $V_\lambda$ with respect to $\br_2$.
Then we may define affine maps $\tau_{\br, i} \colon \h^* = \h^*_{m|n}\oplus \h^*_{m|n} \to \C^{m|n}$   via
$$
\tau_{\br,1}(x,y) = -\tilde{x}-\tilde{\rho}_{\br_1}, 
\quad \text{and}\quad
\tau_{\br,2}(x,y) = %y+\rho_{\br_2}.
\tilde{y}+\tilde{\rho}_{\br_2}.
$$
In this section we prove the following theorem.

\begin{theo}\label{T:1}
Let $\br=\br_1\oplus \br_2$ be an arbitrary Borel subalgebra of $\g$ containing $\h$ and let $\lambda \in   \mathscr{H}(m|n)$. 
Then 
\[
\SPo\circ\tau_{\br,1}(\underline{\lambda}_{\br_1}, \underline{\lambda}_{\br_2})=
\SPo\circ\tau_{\br,2}(\underline{\lambda}_{\br_1}, \underline{\lambda}_{\br_2})=
\SPo\circ\tau_0(-\underline{\lambda}_0,\underline{\lambda}_{0}).
\]
It follows that for any $\br$, the eigenvalue of the Capelli operator  $D^{\mu}$ on each irreducible component $V_\lambda^*\otimes V_\lambda^{}$ of $\p(W)$ is given by evaluating a polynomial on the highest weight of this module with respect to $\br$.  
\end{theo}

\begin{proof} 
Since $\SPo\in\Lambda_{m,n,1}$, it suffices to show that for $i=1,2$ and any $\lambda \in \mathscr{H}(m|n)$, 
\begin{equation}\label{E:toshow}
\tau_{\br,i}(\underline{\lambda}_{\br_1}, \underline{\lambda}_{\br_2})\sim
\tau_0(-\underline{\lambda}_0,\underline{\lambda}_0).
\end{equation}
We verify this in stages.  
First assume that $\br_2$ is an \emph{increasing} Borel subalgebra of $\gl(m|n)$, that is, it is obtained from the standard Borel subalgebra $\br_{m|n}$ by a finite sequence of simple odd reflections. 
Suppose $\br_2 = r_\alpha(\br')$ where $\br'$ is an increasing Borel subalgebra and $\alpha = \epsilon_i-\delta_j$ is a simple isotropic root with respect to $\br'$.  By induction, we may suppose
%Thus it suffices by induction to suppose that $\br_2$ is obtained from an increasing  Borel subalgebra $\br'$ by an odd reflection $r_\alpha$ in  a simple isotropic root $\alpha = \epsilon_i-\delta_j$ with respect to $\br'$, \emph{i.e.},  $\br_2 = r_\alpha(\br')$, and 
that~\eqref{E:toshow} holds for $\br'$, namely, that 
\[
\widetilde{\underline{\lambda}_{\br'}} + \widetilde{\rho_{\br'}} \sim
\widetilde{\underline{\lambda}_0} + \widetilde{\rho_{0}}.
\]
We need to show that 
$\widetilde{\underline{\lambda}_{\br_2}} + \widetilde{\rho_{\br_2}} \sim \widetilde{\underline{\lambda}_{\br'}} + \widetilde{\rho_{\br'}}$.  
Recall from \cite[Lemma 1.40]{cheng_wang_2013} that the highest weight $\underline{\lambda}_{\br_2}$ is obtained from $\underline{\lambda}_{\br'}$ by the formula 
\begin{equation}\label{E:hwchange}
\underline{\lambda}_{\br_2}=
\begin{cases}
\underline{\lambda}_{\br'}-\alpha & \text{if $(\underline{\lambda}_{\br'},\alpha)\neq0$};\\
\underline{\lambda}_{\br'}&\text{if $(\underline{\lambda}_{\br'},\alpha)=0$},
\end{cases}
\end{equation}
while by \cite[Prop 1.33]{cheng_wang_2013} we have $(\rho_{\br'},\alpha)=0$ and $\rho_{\br_2}=\rho_{\br'}+\alpha$.  Therefore if $(\underline{\lambda}_{\br'},\alpha)\neq 0$, we have 
\[
\underline{\lambda}_{\br_2}+\rho_{\br_2}=(\underline{\lambda}_{\br'}-\alpha )+ (\rho_{\br'}+\alpha) = \underline{\lambda}_{\br'}+\rho_{\br'},
\]
and we are done.  Otherwise, we have $(\underline{\lambda}_{\br'},\alpha)=0$, so that
$\underline{\lambda}_{\br_2}=\underline{\lambda}_{\br'}$ and we instead have
\[
\widetilde{\underline{\lambda}_{\br_2}} + \widetilde{\rho_{\br_2}} = \widetilde{\underline{\lambda}_{\br'}} + \widetilde{\rho_{\br'}} + \widetilde{\alpha} =  \widetilde{\underline{\lambda}_{\br'}} + \widetilde{\rho_{\br'}} + e_i-e_{m+j}.
\]
Since $(\underline{\lambda}_{\br'}+\rho_{\br'}, \alpha)=0$, it follows that the vector $(x_1,\cdots, x_m, y_1, \cdots, y_n):=\widetilde{\underline{\lambda}_{\br'}} + \widetilde{\rho_{\br'}}$ satisfies $x_i+y_j=0$.
Thus by Lemma~\ref{lem:equforf} with $\theta=1$ we have
\[
\widetilde{\underline{\lambda}_{\br'}} + \widetilde{\rho_{\br'}}\sim
\widetilde{\underline{\lambda}_{\br'}} + \widetilde{\rho_{\br'}} + e_i-e_{m+j} = \widetilde{\underline{\lambda}_{\br_2}} + \widetilde{\rho_{\br_2}},
\]
as required.  Thus by induction,~\eqref{E:toshow} holds for $i=2$ and all increasing Borel subalgebras $\br_2$.

Now suppose $\br_1$ is a decreasing Borel subalgebra, meaning that $\br_1^\mathrm{op}$ is increasing.  Then by the symmetry of the definition of $\tau_{\br,i}$ we deduce directly that
\begin{align*}
\tau_{\br,1}(\underline{\lambda}_{\br_1}, \underline{\lambda}_{\br_2})&=
\tau_{\br_1\oplus \br_1^{\mathrm{op}},1}\;(\underline{\lambda}_{\br_1}, \underline{\lambda}_{\br_1^{\mathrm{op}}})\\
&=
\tau_{\br_1\oplus \br_1^{\mathrm{op}},2}\;(\underline{\lambda}_{\br_1}, \underline{\lambda}_{\br_1^{\mathrm{op}}})\\
&\sim
\tau_0(-\underline{\lambda}_0,\underline{\lambda}_0),
\end{align*}
which implies \eqref{E:toshow} holds with $i=1$ whenever $\br_1$ is decreasing.

Finally, suppose that $\br=\br_1'\oplus \br_2'$ is an arbitrary Borel subalgebra of $\g$. Then for $i=1,2$ there exist permutations $\sigma_i\in S_m$ (permuting $\epsilon_1,\ldots,\epsilon_m$) and $\sigma_i'\in S_n$ (permuting $\delta_1,\ldots,\delta_n$) such that $(\sigma_i\times\sigma_i') (\br_i')=\br_i$ is a decreasing (respectively, increasing) Borel subalgebra of $\gl(m|n)$ when $i=1$ (respectively, $i=2$).
Then we have for $i=1,2$ that $(\sigma_i\times \sigma_i')(\underline{\lambda}_{\br_i'})=\underline{\lambda}_{\br_i}$ and similarly $(\sigma_i\times \sigma_i') (\rho_{\br_i'}) = \rho_{\br_i}$.  By the separate symmetry of the elements of
 $\Lambda_{m,n,1}$
  these permutations leave polynomials in $\Lambda_{m,n,1}$ invariant, so \eqref{E:toshow} follows immediately.
\end{proof}

\section{Borel subalgebras and highest weights for $\gl(m|2n)$}\label{section:bsam2n}

From now on, we set $\theta:=1/2$, $V:=\C^{m|2n}$, $\g:=\gl(m|2n)$ and   $W:=\s^2(V)$.
Let $\h$ be the standard  Cartan subalgebra of $\g$. Let $\br_0:=\br^{\mathrm{op}}$ be the {\bf opposite} standard (that is, lower triangular) Borel subalgebra of $\g$.

\subsection{The solution to the CEP with respect to $\br_0$}
The 
polynomial algebra $\mathscr{P}(W)$ is a completely reducible and multiplicity-free $\g$-module.  A key result of \cite{sahi_salmasian_serganova_2019} is that we can parametrize this decomposition by $m|n$ hook tableaux.  Given $\lambda \in \mathscr{H}(m|n)$, let $2\lambda\in \mathscr{H}_2(m|2n)$ be the $(m|2n)$ hook tableau obtained by doubling each row, and let $W_\lambda:=(V_{2\lambda})^*$. Then we have the decomposition
\begin{equation}\label{eq:decompofp}
\mathscr{P}(W)=\bigoplus_{\lambda \in \mathscr{H}(m|n)}W_\lambda.
\end{equation}

In this case, given $\lambda=(\lambda_1,\lambda_2, \ldots)\in \mathscr{H}(m|n)$ with transpose $\lambda'=(\lambda_1', \lambda_2', \ldots)$, the highest weight $\underline{\lambda}_0$ of $W_\lambda$ with respect to $\br_0=\br_{m|2n}^{\mathrm{op}}$ is given by
\begin{equation}\label{eq:hwindecompm2n}
\underline{\lambda}_0=-\sum_{i=1}^m2\lambda_i\epsilon_i-\sum_{j=1}^n\mu_j\pr{\delta_{2j-1}+\delta_{2j}}
\end{equation}
where $\mu_j:=\langle\lambda'_j-m\rangle=\max\{0,\lambda'_j-m\}$.  Note that this corresponds precisely to the $m|2n$ hook partitition of $2\lambda$, which is $(2\lambda_1, \dots, 2\lambda_m|\mu_1,\mu_1, \dots, \mu_n,\mu_n) \in \mathscr{H}_2(m|2n)$.

Let $\Omega_{m|2n} = \{\underline{\lambda}_0 \mid \lambda \in \mathscr{H}(m|n)\}$ be the set of all $\br_0$-highest weights of $\g$-modules appearing in the decomposition \eqref{eq:decompofp}. The Zariski closure of  $\Omega_{m|2n}$ in $\h^*$ is a   subspace we denote  $\mathfrak{a}^*$.
%, since it  coincides with the subspace of  linear functionals on $\mathfrak h$ that vanish on $\mathfrak t:=\mathfrak h\cap \mathfrak k$.
 Let $\weird \colon \mathfrak{h}^*\to \C^{m|n}$ be the linear map defined by
$$
\weird(a_1, \ldots, a_m|b_1, \ldots, b_{2n}) = \left(\frac{a_1}{2}, \cdots, \frac{a_m}{2}\mid \frac{b_1+b_2}{2}, \cdots, \frac{b_{2n-1}+b_{2n}}{2}\right).
$$
Observe that if $\rho_0$ denotes the Weyl vector corresponding to $\br_0$, then we have
\begin{equation}\label{eq:glm2nstd}
\weird(\rho_0) =  -\sum_{i=1}^m\frac{m-2n+1-2i}{4}e_i - \sum_{k=1}^n\frac{m+2n+2-4k}{2}e_{m+k}.
\end{equation}
Define, for any $x\in \mathfrak{a}^*$, the map $\tau_0^{\mathfrak{a}^*}\colon \mathfrak{a}^* \to \C^{m|n}$ by $\tau_0^{\mathfrak{a}^*}(x)=-\weird(x + \rho_0)$.  Then it is easy to see that 
we can rephrase the result of  \cite[Table 3]{sahi_salmasian_serganova_2019} pertaining to $\g$ as follows.

\begin{theo}\cite[Theorem 1.13.]{sahi_salmasian_serganova_2019}\label{hadi_theo_2}
Let $\g$ and $W$ be as above.  
 Then for each $\lambda,\mu\in\mathscr{H}(m|n)$,  the eigenvalue of the Capelli operator $D^\mu$  on the irreducible summand $W_\lambda$  with highest weight $\underline{\lambda}_0\in \mathfrak{a}^*$ with respect to $\br_0$ is equal to  
 $
 {P}_{\mu,\frac12} \circ \tau_0^{\mathfrak{a}^*}({\underline{\lambda}_0}).
 $
\end{theo}

\subsection{Explicit form of $\tau_0(\underline{\lambda}_0)$}

The affine map $\tau^{\mathfrak{a}^*}_0$ is only characterized by the theorem  on elements of $\mathfrak{a^*}$.  To explicitly describe all extensions of $\tau^{\mathfrak{a}^*}_0$  to affine maps  $\tau_0\colon \h^*\to\C^{m|n}$, let $M_0$ denote the matrix of the linear transformation $-\weird \colon \h^*\to \C^{m|n}$ with respect to our standard bases.  Then  we have
\begin{equation}\label{E:M0}
M_0 = \mat{-\frac12 I_m & 0_{m\times 2n} \\ 0_{n\times m} & D} 
\end{equation}
where $D = (d_{ij})$ is a $n\times 2n$ matrix whose entries satisfy
\[
d_{ij} = \begin{cases} 
-\frac12 & \text{if } j=2i-1\mbox{ or }2i \\ 
0 & \text{otherwise}.
\end{cases}
\]
\begin{defn}\label{def:setofM}
Let $\mathcal{C}$ be the set of all matrices of the form
$$
M = M_0 + \mat{0_{(m+n)\times m} & A}
$$
where 
 $A$ is any $(m+n)\times (2n)$ matrix whose columns $A_j$ satisfy the relations
$A_{2k} = -A_{2k-1}$
for all $1\leq k \leq n$.  We also set 
\begin{equation}\label{E:X0rho}
X_0 = -\weird(\rho_0)= M_0\rho_0 = \tau_0^{\mathfrak{a}^*}(0).
\end{equation}
\end{defn}

\begin{lem}\label{lem:Mform}
Every extension of $\tau_0^{\mathfrak{a}^*} \colon \mathfrak{a}^*\to \C^{m|n}$ to an affine map on $\h^*$ is given on $x\in \h^*$ by $\tau_0(x)=Mx+X_0$ for some $M\in \mathcal{C}$.
\end{lem}

\begin{proof}
It suffices to observe that $\mathfrak{a}^*=\text{span}\{\epsilon_1, \cdots, \epsilon_m, \delta_1+\delta_2, \cdots, \delta_{2n-1}+\delta_{2n}\}$; thus the set of all linear maps to $\C^{m|n}$ vanishing on $\mathfrak{a}^*$
is given by
\[
\mathcal{K}=\left\{ \mat{0_{(n+m)\times m} & A}  \mid A_{2k-1}=-A_{2k} \; \text{ for all $1\le k\le n,$}\right\}
\]
where $A_j$ denotes the $j$th column of $A$.
\end{proof}

Note that since $\rho_0\notin \mathfrak{a}^*$, the equation  $M\rho_0 = X_0$ with  $M\in \mathcal{C}$ holds \emph{only} for $M=M_0$.  In the following sections, we will see that the single matrix $M_0$ is not suitable for use with other Borel subalgebras, but that other subsets of $\mathcal{C}$ will be.

\subsection{Borel subalgebras and decreasing $\delta\epsilon$ sequences}\label{subsec:decreasingb}

One may parametrize the set $\mathcal{B}$ of Borel subalgebras of $\g$ containing $\h$ by sequences $(\xi_i)_{i=1}^{m+2n}$ where $\{\xi_i |1\leq i\leq m+2n\} = \{ \epsilon_i, \delta_j\mid 1\leq i \leq m, 1\leq j\leq 2n\}$ --- called $\delta\epsilon$ sequences --- by associating to $\br$ the sequence for which $\{\xi_{i}-\xi_{i+1}\mid 1\leq i < m+2n\}$ is its set of simple roots.  In particular, the simple odd (isotropic) roots of $\br$ correspond to pairs $\epsilon_i\delta_j$ or $\delta_j\epsilon_i$ in the  sequence, and if $\alpha$ is such a root, then the $\delta\epsilon$ sequence of $r_\alpha(\br)$ is obtained from that of $\br$ by swapping $\epsilon_i$ and $\delta_j$. 

For example, $\br^{\mathrm{op}}$ corresponds to the sequence $\delta_{2n}\delta_{2n-1}\cdots \delta_1\epsilon_m \epsilon_{m-1}\cdots \epsilon_1$ and it has a single odd simple root $\delta_1-\epsilon_m$.  

 If the indices of the $\epsilon_i$ and of the $\delta_j$ are (separately) decreasing, then we say this is a decreasing $\delta\epsilon$-sequence.  These correspond exactly to the decreasing Borel subalgebras defined earlier, since each one may be obtained from $\br^{\mathrm{op}}$ be a sequence of simple odd reflections.  Let $\mathcal{B}^d$ denote the set of decreasing Borel subalgebras of $\g$ containing $\h$.

\begin{defn}\label{D:ellj}
Let $\br\in\mathcal{B}^d$, and consider the $\delta\epsilon$ sequence associated to $\br$.
\begin{enumerate}
\item For each $1\leq i \leq m$,  let $\ell_{i,\br}$ be the number of $\delta$s to the right of $\epsilon_i$.
\item For each $1 \leq k \leq 2n$, let $j_{k,\br}$ be the number of $\epsilon$s to the left of $\delta_k$.  
\end{enumerate}
\end{defn}

When $\br$ is clear from the context, we may write $\ell_i$ and $j_k$ for $\ell_{i,\br}$ and  $j_{k,\br}$ respectively.  It is easy to see that, for example, $j_k = \#\{i \mid \ell_{i}\geq k\}$.  Thus a decreasing Borel subalgebra (or equivalently, the corresponding decreasing $\delta\epsilon$ sequence) is completely determined by specifying the values $\ell_1 \leq \cdots \leq \ell_m$ (or equivalently, the values $j_1\geq  \cdots \geq j_{2n}$). 

\begin{defn}\label{defn:genericroots}
Let $\br\in\mathcal{B}^d$. With $\ell_i$ as defined above, let $\mathcal{G}_{\br}$ be the set of roots  
\begin{align}  %\label{1*}
\label{E:star}\qquad &\delta_1-\epsilon_m, \delta_2-\epsilon_m, \cdots, \delta_{\ell_m}-\epsilon_m; \\
&\delta_1-\epsilon_{m-1}, \delta_2-\epsilon_{m-1}, \cdots, \delta_{\ell_{m-1}}-\epsilon_{m-1};\nonumber\\
&\cdots\nonumber\\
&\delta_1-\epsilon_1, \delta_2-\epsilon_1, \cdots, \delta_{\ell_1}-\epsilon_1.\nonumber
\end{align}
\end{defn}

The importance of $\mathcal{G}_{\br}$ is as follows.  Suppose that we obtain $\br$ from $\br^{\mathrm{op}}$ by a sequence of odd simple reflections, that is, such that $\br_0 = \br^{\mathrm{op}}$, $\br=\br_t$ for some $t>0$ and for each $1\leq k \leq t$, 
$\br_k = r_{\alpha_k}(\br_{k-1})$ where $\alpha_k$ is an odd simple (isotropic) root of $\br_{k-1}$. 
Then if $t$ is minimal we necessarily have the equality of sets $\{\alpha_i \mid 1\leq i \leq t\}=\mathcal{G}_{\br}$, since the effect of each simple odd reflection $\pm(\epsilon_i-\delta_j)$  is to swap the adjacent terms $\epsilon_i$ and $\delta_j$ in the corresponding $\delta\epsilon$ sequence.  Moreover, the ordered list of roots in \eqref{E:star}, (ordered from the top row to the bottom row) is one sequence satisfying these conditions. 
Set
\begin{equation}\label{E:rbr}
r_{\br}:=\sum_{\alpha\in\mathcal{G}_\br} \alpha.
\end{equation} 
The following lemma is immediate. 
\begin{lem}\label{L:rbr}
For each $\br\in \mathcal{B}^d$ we have
$$
r_{\br} = -\sum_{i=1}^m \ell_{i,\br}\epsilon_i  + \sum_{k=1}^{2n} j_{k,\br}\delta_{k}.
$$
\end{lem}

\subsection{A formula for highest weights with respect to $\br$}\label{sub:formula}

Let $\lambda\in \mathscr{H}(m|n)$.  
For any $\br\in \mathcal{B}^d$ (respectively, for $\br^{\mathrm{op}}$) write $\underline{\lambda}_{\br}$ (respectively, $\underline{\lambda}_0$) for the corresponding highest weight of $W_\lambda$.

Suppose $\br$ is  obtained from $\br^{\mathrm{op}}$ by the sequence of simple odd reflections in the roots \eqref{E:star}.   For any adjacent pair of Borel subalgebras $\br'' = r_{\delta_j-\epsilon_i}(\br')$ of this sequence, we have from \eqref{E:hwchange} that 
$$
\underline{\lambda}_{\br''} = \underline{\lambda}_{\br'}- \delta_j+\epsilon_i 
$$
unless $(\underline{\lambda}_{\br'},\delta_j-\epsilon_i)=0$.  Since all coefficients of $\underline{\lambda}_{\br'}$ are nonpositive, this latter holds if and only if its $i$th and $(m+j)$th coefficients are both zero.

It follows that 
the highest weight of $W_\lambda$ with respect to $\br$ takes the form
$$
\underline{\lambda}_{\br} = \underline{\lambda}_0-r({\lambda},\br), \quad \text{where} \quad r({\lambda},\br):=\sum_{\alpha \in \mathcal{N}_{\br}}\alpha
$$
for some subset $\mathcal{N}_{\br}\subseteq \mathcal{G}_{\br}$.

\begin{defn}
If $\mathcal{N}_{\br}=\mathcal{G}_{\br}$, then $r({\lambda},\br) = r_{\br}$ and we say that  $\lambda\in \mathscr{H}(m|n)$ % the highest weight $\underline{\lambda}_0\in \Omega_{m|2n}$ 
is \emph{generic} for $\br$.   Otherwise, 
%$\underline{\lambda}_0$ 
$\lambda$ is called nongeneric for $\br$.
\end{defn}

Note that since $(\ell_{i,\br})_i$ is increasing  whereas $(2\lambda_i)_i$ is decreasing, these sequences cross at most once.  The  indices $i$ for which $\ell_{i,\br}\leq 2\lambda_i$ are key to describing $r({\lambda},\br)$ in the nongeneric case.

\begin{defn}\label{defnimportantI}Let $\br\in\mathcal{B}^d$ and $\lambda\in\mathscr{H}(m|n)$. If $\ell_{m,\br}>2\lambda_m$, then let $I_{\lambda,\br}$ denote the least index $I\geq 1$ for which $
\ell_{I}> 2\lambda_{I}$.
Further, set 
$$
\ell(\lambda, i,\br) :=\min\{\ell_{i,\br},2\lambda_i\}
$$ for all $1\le i\le m$; then $\ell(\lambda, i,\br)=\ell_{i,\br}$ if and only if $i < I_{\lambda,\br}$.
\end{defn}

\begin{prop}\label{P:highestweightMN}
Let $\br\in \mathcal{B}^d$ and $\lambda\in \mathscr{H}(m|n)$.    Then $r(\lambda,\br) = \sum_{\alpha \in \mathcal{N}_\br}\alpha$ where $\mathcal{N}_\br$ is the following set of roots:
\begin{align}
\label{E:rlambdab}\qquad &\delta_1-\epsilon_m, \delta_2-\epsilon_m, \cdots, \delta_{\ell({\lambda}, m,\br) }-\epsilon_m; \\
&\delta_1-\epsilon_{m-1}, \delta_2-\epsilon_{m-1}, \cdots, \delta_{\ell({\lambda}, m-1,\br) }-\epsilon_{m-1};\nonumber\\
&\cdots\nonumber\\
&\delta_1-\epsilon_1, \delta_2-\epsilon_1, \cdots, \delta_{\ell({\lambda}, 1,\br) }-\epsilon_1,\nonumber
\end{align}
where we interpret any row corresponding to a value of $\ell(\lambda, i,\br)=0$ as an empty row. %In particular,
In particular $\underline{\lambda}_{\br}=\underline{\lambda}_0-r({\lambda},\br)$ is the $\br$-highest weight of $W_\lambda$.  
\end{prop}

\begin{proof} 
We have already addressed the case that 
%$\underline{\lambda}_0$ 
$\lambda$ is generic for $\br$, so suppose 
%$\underline{\lambda}_0$ 
$\lambda$ is nongeneric for $\br$ and let $I:=I_{ {\lambda},\br}$.
It suffices to show that $\underline{\lambda}_{\br} = \underline{\lambda}_0 - r({\lambda},\br)$.  Note  that in $\underline{\lambda}_0=(2\lambda_1, \cdots, 2\lambda_m|\mu_1,\mu_1, \cdots, \mu_n,\mu_n)$, we have $\mu_k=0$ for all $k>\lambda_m$.

We proceed inductively from $m$ down to $1$, following the rows of the list of roots \eqref{E:star}, showing that if $\br_{j+1}$ denotes the Borel subalgebra obtained after completing all reflections up to the end of those in the row corresponding to  $\epsilon_{j+1}$, that
\begin{equation}\label{eq:induction}
\underline{\lambda}_{\br_{j+1}} = \underline{\lambda}_0 - \sum_{i=j+1}^m (\delta_1+\cdots+\delta_{\ell(\lambda, i, \br)} - \ell(\lambda, i, \br) \epsilon_{i}).
\end{equation}
The base case corresponds to $j=m$, $\br=\br^{\mathrm{op}}$, where there is nothing to show.

Suppose we have shown our inductive hypothesis up to $\epsilon_{j+1}$, as in \eqref{eq:induction}.  Let us first determine the coefficients of $\epsilon_j$ and of $\delta_k$ with $k>2\lambda_j$.

The coefficient of $\epsilon_{j}$ in $-\underline{\lambda}_{\br_{j+1}}$ is the same as that of $-\underline{\lambda}_0$; this value is therefore $2\lambda_j$.   Let $k>2\lambda_j \geq 2\lambda_m$.  Then the coefficient of $\delta_k$ in $\underline{\lambda}_0$ was zero.  Moreover, since we have  
$2\lambda_j \geq 2\lambda_i\geq \ell(\lambda, i,\br)$ for each $j+1\leq i \leq m$, none of the additional terms in \eqref{eq:induction} contribute to the coefficient of $\delta_k$.  Thus the coefficient of $\delta_k$ in $-\underline{\lambda}_{\br_{j+1}}$ is $0$ for $k>2\lambda_j$.

As we iterate through each of the roots
$$
\delta_1-\epsilon_j, \delta_2-\epsilon_j, \cdots, \delta_{\ell(\lambda, j, \br)}-\epsilon_j
$$
we fall into the first case of \eqref{E:hwchange} because the coefficient of $\epsilon_j$ is nonzero at each iteration.  Therefore if $\br'$ denotes the resulting Borel subalgebra, we have
\begin{equation}\label{E:stepind}
\underline{\lambda}_{\br'} = \underline{\lambda}_{\br_{j+1}} - (\delta_1+\cdots+\delta_{\ell(\lambda, j, \br)} - \ell(\lambda,j,\br) \epsilon_{j}).
\end{equation}
If $\ell(\lambda,j,\br)=\ell_{j,\br}$ (that is, if $j\leq I$) then we are done.
Otherwise, $\ell(\lambda,j,\br)=2\lambda_j$ and we infer that the coefficient of $\epsilon_j$ in $\underline{\lambda}_{\br'}$ is $0$, as is the coefficient of $\delta_{k}$ for any $k>\ell(\lambda,j,\br) = 2\lambda_j$. 
Thus for any remaining roots 
$$
\delta_{\ell(\lambda,j,\br)+1}-\epsilon_j, \cdots, \delta_{\ell_{j,\br}}-\epsilon_j
$$
we fall into the second case of \eqref{E:hwchange}, and so $\underline{\lambda}_{\br_j}=\underline{\lambda}_{\br'}$.  Therefore \eqref{eq:induction} and \eqref{E:stepind} together imply
$$
\underline{\lambda}_{\br_j}= \underline{\lambda}_0 - \sum_{i=j}^m (\delta_1+\cdots+\delta_{\ell(\lambda, i, \br)} - \ell(\lambda, i, \br) \epsilon_{i}),
$$
and we may conclude by induction that
\begin{equation}\label{E:difflambda}
\underline{\lambda}_{\br} = \underline{\lambda}_0 - \sum_{i=1}^m (\delta_1+\cdots+\delta_{\ell(\lambda, i, \br)} - \ell(\lambda, i, \br) \epsilon_{i}) = \underline{\lambda}_0-r({\lambda},\br),
\end{equation}
as required.
\end{proof}

\begin{remark}\label{Rem:coeffszero}
Note in particular that \eqref{E:difflambda} implies that the coefficient of $\epsilon_i$ in $\underline{\lambda}_\br$ is given by
$$
-(2\lambda_i - \ell(\lambda,i,\br)) = \begin{cases}
-(2\lambda_i - \ell_i) <0 & \text{if $\ell_i< 2\lambda_i$ ;}\\
0 & \text{if $\ell_i\geq 2\lambda_i$.}
\end{cases}
$$
The coefficient of $\epsilon_i$ is always $\leq 0$.
\end{remark}

\begin{cor}
\label{prp:generic-char}
Let $\br\in\mathcal B^d$ and $\lambda\in\mathscr H(m|n)$. The following statements are equivalent.
\begin{itemize}
\item[\rm(i)] $\lambda$ is $\br$-generic.
\item[\rm (ii)]
$2\lambda_m\geq \ell_{m,\br}$.
\item[\rm (iii)]  $2\lambda_i\geq \ell_{i,\br}$ for $1\leq i\leq m$.
\end{itemize}
\end{cor}

\begin{proof}
Equivalence of (ii) and (iii) is an immediate consequence of the fact that the $\ell_{i,\br}$ are increasing whereas the $\lambda_i$ are decreasing. The equivalence of (i) and (ii) 
follows from Proposition~\ref{P:highestweightMN} and the fact that 
 if $2\lambda_i\geq \ell_{i,\br}$ for all $i$, then $r(\lambda,\br)=r_\br$ so $\lambda$ is generic for $\br$.
%Suppose $\underline{\lambda}_0 = -(2\lambda_1, \cdots, 2\lambda_m|\mu_1, \mu_1, \cdots, \mu_n,\mu_n)$.  If we apply the odd reflection associated to $\beta:=\delta_1-\epsilon_m$ then $\underline\lambda_0$ will be replaced by $\underline\lambda_0-\beta$ if and only if $(\underline\lambda_0,\beta)\neq 0$.
%But 
%\[(\underline \lambda_0,\beta)=\mu_1+2\lambda_m.\]
%Since $2\lambda_m=0$ implies that $\mu_1=0$, the condition 
%$(\underline\lambda_0,\beta)\neq 0$ is equivalent to $2\lambda_m\geq 1$. More generally, the odd reflections associated to $\delta_j-\epsilon_m$ for $1\leq j\leq \ell_m$ result in subtractions of the corresponding odd roots if and only if  $2\lambda_m\geq \ell_{m,\br}$. Furthermore, once the latter condition is satisfied we also have $2\lambda_i\geq\ell_{i,\br}$ (because the $\lambda_i$ are decreasing whereas the $\ell_{i,\br}$ are increasing) and consequently all of the odd reflections for the roots in $\mathcal G_\br$ result in substraction of their odd roots. 
%It is now immediate from Lemma~\ref{L:rbr} and the discussion above that  $\underline{\lambda}_0$ is generic for $\br$ if and only if  $\ell_{m,\br}\leq 2\lambda_m$.
\end{proof}

For ease of reference, we list the roots contributing to $r_{\br}-r({\lambda},\br)$, that is, the subset $\mathcal{G}_\br\setminus \mathcal{N}_\br$  of the set of roots in \eqref{E:star}
which we will \emph{not} subtract from $\underline{\lambda}_0$ when computing the highest weight $\underline{\lambda}_{\br}$.  Using  the index $I=I_{\lambda,\br}$ for simplicity, it follows from \eqref{E:rlambdab} that they are
\begin{align}
\label{E:doublestar} \qquad &\delta_{2\lambda_m+1}-\epsilon_m, \cdots , \delta_{\ell_{m,\br}}-\epsilon_m, \\
&\delta_{2\lambda_{m-1}+1}-\epsilon_{m-1}, \cdots , \delta_{\ell_{m-1,\br}}-\epsilon_{m-1},\nonumber \\
&\cdots\nonumber\\
&\delta_{2\lambda_{I}+1}-\epsilon_I, \cdots , \delta_{\ell_{{I },\br}}-\epsilon_I\nonumber,
\end{align}

\subsection{Some necessary conditions}
Our goal is to define, for each $\br\in \mathcal{B}^d$, an affine map $\tau_\br \colon \h^* \to \C^{m|n}$ such that for all $\lambda \in \mathscr{H}(m|n)$ we have the equivalence
\begin{equation}\label{eq:whatwewant}
\tau_\br(\underline{\lambda}_{\br}) \sim \tau_0(\underline{\lambda}_0).
\end{equation}
As in Section~\ref{sec:CEPdiag}, this would imply that  
$P_{\lambda,1/2}\circ \tau_\br$ is the interpolation polynomial we seek. 
Such a map $\tau_\br$ will be given on $x\in\h^*$ by 
$$
\tau_\br(x) = M_\br x + X_\br,
$$ 
for some $(m+n)\times(m+2n)$ matrix $M_\br$ and vector $X_\br\in \C^{m|n}$.  In this section derive some necessary conditions that $\tau_\br$ must satisfy in order for the relation \eqref{eq:whatwewant} to hold.

\begin{lem} \label{E:XsimX0}
Let $\br\in \mathcal{B}^d$.  If \eqref{eq:whatwewant} holds, then we must have 
$$
X_\br \sim X_0.
$$
\end{lem}

\begin{proof}
Since $\underline{\lambda}_0=\underline{\lambda}_\br=0$ when $\lambda$ corresponds to the trivial module,  the requirement that 
$\tau_\br(\underline{\lambda}_{\br}) \sim \tau_0(\underline{\lambda}_0)$ for all $\lambda$ implies in particular that  
$X_{\br}\sim X_0$.
\end{proof}

\begin{prop}\label{Prop:equalgeneric}
Let $\br\in \mathcal{B}^d$.  If \eqref{eq:whatwewant} holds, then for all $\lambda\in \mathscr{H}(m|n)$  that are $\br$-generic we must have the equality
\begin{equation} \label{eq:equalgeneric}
M_\br\underline{\lambda}_{\br}+X_\br= M_0\underline{\lambda}_0+X_0.
\end{equation}
In consequence, we must have 
$$
M_\br \in \mathcal{C} \quad \text{and} \quad M_\br r_{\br} =X_\br-X_0.
$$
\end{prop}

\begin{proof}
By \eqref{E:X0rho}, the right hand side of \eqref{eq:whatwewant} is equal to $M_0(\underline{\lambda}_0+\rho_0)$.  On the other hand, the definition of genericity implies $\underline{\lambda}_\br=\underline{\lambda}_0-r_\br$ so the left hand side is
$$
M_\br\underline{\lambda}_\br + X_{\br} = M_\br\underline{\lambda}_0 + (-M_\br r_\br+X_\br).
$$
We first show that the equivalence of the two sides can only be achieved by equality.

Namely, recall that  
%$\underline{\lambda}_0$ 
$\lambda$ is generic for a fixed $\br\in \mathcal{B}^d$ only if $2\lambda_i\geq \ell_{i,\br}$ for all $i$.  Therefore we may freely choose generic $\lambda$ sufficiently large such that the vector 
$M_0(\underline{\lambda}_0+\rho_0)$ does not lie on any hyperplane of the form $x_i+\frac12y_j=\pm1/4$, where a symmetry of the form $\sim$ holds (with $\theta=1/2$) and therefore no symmetry of the form \eqref{eq:mosys} can apply. 
 Similarly, the fact that $X_0 \neq 0$, and that equivalence must hold for the $\br$-generic weights $s\underline{\lambda}_0$  for every $s\in \mathbb{Z}_{\geq 2}$, excludes the possibility that the equivalence is due to separate symmetry of the variables. Therefore, for all sufficiently large $\br$-generic $\lambda$, the equivalence \eqref{eq:whatwewant} must occur as an equality, that is, we have
$$
M_\br \underline{\lambda}_0 + (-M_\br r_\br+X_\br) = M_0\underline{\lambda}_0+X_0.
$$
Since
$$
\text{span}\{\underline{\lambda}_0\in \Omega_{m|2n} \mid \lambda \text{ is generic for $\br$ and } \underline{\lambda}_0 \gg 0\} = \mathfrak{a}^*,
$$
we deduce that this equality is in fact an equality of affine functions on $\mathfrak{a}^*$.  Thus we have $X_0=-M_\br r_\br +X_\br$ and, by Lemma~\ref{lem:Mform}, we conclude $M_\br  \in \mathcal{C}$.
\end{proof}

\section{The  refined  CEP for $\pr{\gl(V),\s^2\left(V\right)}$}\label{sec:cpeglm2n}

We develop our solution to the refined CEP in stages of increasing complexity of the Borel subalgebra.

\subsection{The very even case}  \label{SS:veryeven}  We begin with the set of Borel subalgebras $\br$ that have the property that $\underline{\lambda}_\br\in \mathfrak{a}^*$ for all $\lambda \in \mathscr{H}(m|n)$ : the ones for which for all $k$, the terms $\delta_{2k}\delta_{2k-1}$ are adjacent in the $\delta\epsilon$ sequence. 

\begin{defn}\label{defn:veryeven}
A decreasing Borel subalgebra $\br$ and its associated $\delta\epsilon$ sequence are called \emph{very even} if $\ell_{i,\br}$ is even for all $1\leq i\leq m$; equivalently, if $j_{2k-1}=j_{2k}$ for all $1\leq k \leq n$.
\end{defn}

Suppose $\br$ is very even.  For any $\lambda \in \mathscr{H}(m|n)$, the value $\ell(\lambda,i,\br) = \min\{\ell_{i,\br},2\lambda_i\}$ will also be even.  
Since 
$\mathfrak{a}^* = \text{span}\{ \delta_{2k-1}+\delta_{2k},2\epsilon_i \mid 1\leq k \leq n, 1\leq i\leq m\}$, it is immediate that the roots comprising $r(\lambda,\br)$  can be partitioned into pairs whose sums lie  in this span, and so by Proposition~\ref{P:highestweightMN} we have $\underline{\lambda}_\br \in \mathfrak{a}^*$.
Henceforth we write $\rho_\br$ for the Weyl vector of $\br$. 

\begin{lem}\label{L:Xbr}
Suppose $\br$ is very even and $\tau_\br \colon \h^*\to \C^{m|n}$ is an affine map satisfying \eqref{eq:whatwewant}, in the form $\tau_\br(x)=Mx+X_\br$ for all $x\in \h^*$.   Then $X_\br = M_0\rho_\br$ 
where $M_0$ is as in \eqref{E:M0}, and in fact, $\tau_\br$ is independent of the choice of $M\in \mathcal{C}$.
\end{lem}
%\begin{lem}\label{L:Xbr}
%If $\br$ is very even and $\tau_\br \colon \h^*\to \C^{m|n}$ satisfies \eqref{eq:whatwewant}, then for all $x\in \h^*$ we must have
%$$
%\tau_\br(x) = Mx + X_\br
%$$
%for some $M\in \mathcal{C}$, where $X_\br = Mr_\br+X_0$. {\color{magenta}
%Conversely, given any $M\in\mathcal C$, the corresponding map $\tau_\br$ satisfies~\eqref{eq:whatwewant}. 
%}  \textcolor{blue}{Don't we need to prove this below?} Moreover, $X_\br = M_0\rho_\br$ where $M_0$ is as in \eqref{E:M0}. {\color{magenta}
%In particular, $X_\br$ is independent of the choice of $M$. 
%}\end{lem}
Note that since $\rho_\br\notin \mathfrak{a}^*$, it is not true that $M\rho_\br = X_\br$ for all $M\in \mathcal{C}$.

\begin{proof}
Suppose $\br$ is very even and suppose $\tau_\br(x)=Mx+X$ is an affine function satisfying \eqref{eq:whatwewant}.  By Proposition~\ref{Prop:equalgeneric}, we must have $M\in\mathcal{C}$ and $X = Mr_\br+X_0$.  If $M' \in \mathcal{C}$, then $M-M'$ is zero on $\mathfrak{a}^*$.  Thus $M\underline{\lambda}_\br = M'\underline{\lambda}_\br$ for every $\lambda\in \mathscr{H}(m|n)$ and $M'r_\br=Mr_\br$, so if one choice of $M\in\mathcal{C}$ gives a function satisfying \eqref{eq:whatwewant}, so do all.  Further, since by \cite[Prop 1.33]{cheng_wang_2013} we have $\rho_\br = \rho_0+r_\br$, we readily compute $M_0\rho_\br = M_0(\rho_0+r_\br) = X_0 +M_0r_\br=X_\br$.  
\end{proof}

The rest of this subsection is devoted to proving the converse, that is, that
for all very even $\br$,
the affine functions
$
\tau_\br(x)=Mx+M_0\rho_\br
$ for $M\in\mathcal C$ 
 satisfy~\eqref{eq:whatwewant}. We first prove our key technical result, which gives a collection of terms that are  equivalent to $\tau_0(\underline{\lambda}_0)$.  By Lemma~\ref{L:Xbr}, %the uniqueness shown in the lemma above, 
 it suffices to prove the result for $M=M_0$.

\begin{prop}\label{P:equiveven}
Suppose $\br$ is very even and that $\underline{\lambda}_0\in \Omega_{m|2n}$.  Then  we have
$$
M_0\underline{\lambda}_0+X_0 \sim M_0(\underline{\lambda}_0+r_{\br}-r(\lambda,\br))+X_0.
$$
\end{prop}

For ease of reference in the proof, for all $1\le i\le m$ and $1\le j\le n$, let 
\begin{equation}\label{EandFglm2n}
E_i \coloneqq\frac{m+1-2n-2i}{4} \mbox{ and }F_k \coloneqq\frac{m+2+2n-4k}{2}.
\end{equation} 
Then we have $X_0 = M_0\rho_0=\sum_{i=1}^m E_i e_i + \sum_{k=1}^n F_k e_{m+k}$.

\begin{proof}
Suppose $\br$ is very even and let $I$ be the least index for which $\ell_{I,\br}>2\lambda_I$; then we must have $\ell_{I,\br} \geq 2\lambda_I+2$.   Considering moreover the list of roots occuring in \eqref{E:doublestar}, whose sum is $r_\br -r(\lambda,\br)$, we note that 
$$
\lambda_m+1 \leq \lambda_{m-1}+1 \leq \cdots \leq \lambda_I+1 \leq \frac12\ell_{I,\br} \leq \frac12 \ell_{I+1,\br} \leq \cdots \leq \frac12 \ell_{m,\br}.
$$
That is, the intervals 
$(\lambda_i+1, \ell_{i,\br}/2)$ are nested, as $i$ decreases.

To prove the proposition, we proceed by induction on $i$ from $m$ down to $I$ on the rows of \eqref{E:doublestar}.  Let
$$
r_{\br,\lambda,j} = \left( \sum_{k=2\lambda_{j}+1}^{\ell_{j,\br}} \delta_{k}\right) - (\ell_{j,\br}-2\lambda_{j})\epsilon_{j}
$$
denote the sum of the roots on the row of \eqref{E:doublestar} corresponding to $\epsilon_j$.  Set
$$
\Lambda_{i+1} = M_0(\underline{\lambda}_0+\sum_{j=i+1}^m r_{\br,\lambda,j})+X_0,
$$
which is the partial sum obtained after completing the row indexed by $i+1$.

We will show that we can transform $\Lambda_{m+1}:=M_0\underline{\lambda}_0 + X_0$, by a sequence of monoidal symmetries as in Lemma~\ref{lem:equforf},
into the final result
$$
\Lambda_{I} = M_0(\underline{\lambda}_0+r_{\br}-r(\lambda,\br)) +X_0.
$$
Our inductive hypothesis is that 
$$\Lambda_{i+1} \sim \Lambda_i,$$  that 
the coefficient of $e_i$ in $\Lambda_{i+1}$ is $\lambda_{i} + E_i$ and that
the coefficients of $e_{m+j}$ in $\Lambda_{i+1}$ with $\lambda_{i}+1 \leq j \leq \frac12\ell_{i,\br}$ are given by $F_j - (m-i)$.

The base case is that $i=m$, so equivalence is a tautology and the statement about the coefficients is true.
Suppose it holds for $i+1$.  
Let us prove the statement for $i$.   Note that 
\begin{align*}
\Lambda_{i} &= \Lambda_{i+1} + M_0r_{\br,\lambda,i}\\
&= \Lambda_{i+1} + M_0\left(\pr{\sum_{t=2\lambda_{i}+1}^{\ell_{i,\br}} \delta_{t}} - (\ell_{i,\br}-2\lambda_{i})\epsilon_{i}\right)\\
&= \Lambda_{i+1} + M_0\left(\sum_{k=\lambda_{i}+1}^{\ell_{i,\br}/2} (\delta_{2k-1} + \delta_{2k} - 2\epsilon_{i})\right)\\
&= \Lambda_{i+1} + \sum_{k=\lambda_i+1}^{\ell_{i,\br}/2}(e_i-e_{m+k}).
\end{align*}
Suppose we have shown for some $K$, $\lambda_i+1\leq K <\ell_i/2$, that $\Lambda_{i+1}$ is  monoidally equivalent to
$$
\Lambda_{i+1,K-1}= \Lambda_{i+1} + \sum_{k=\lambda_i+1}^{K-1} (e_i-e_{m+k}).
$$
We now show that $\Lambda_{i+1,K-1}\sim \Lambda_{i+1,K}$. Then by induction as $K$ runs from $\lambda_i+1$ to $\ell_{i,\br}/2$, we will conclude that $\Lambda_i\sim\Lambda_{i+1}.$

Using the outer induction hypothesis, as well as the formula for $\Lambda_{i+1,K-1}$, we compute that the coefficient of $e_i$ in $\Lambda_{i+1,K-1}$ is $\lambda_i+E_i+(K-\lambda_i-1)$ and the coefficient of $e_{m+K}$ in $\Lambda_{i+1,K-1}$ is $F_K-(m-i)$.
Since in this case we have $\theta=1/2$, Lemma~\ref{lem:equforf} says $y\sim y+e_i-e_{m+k}$ if $y_i+\frac12 y_{m+k}=-\frac14$.
Using \eqref{EandFglm2n}, we compute
$$
\lambda_i+E_i+\pr{K-\lambda_i-1} + \frac12\pr{F_K-(m-i))} = -\frac14
$$
Therefore monoidal symmetry may be applied and we conclude that
$$
\Lambda_{i+1,K-1} \sim \Lambda_{i+1,K-1} + e_{i} - e_{m+K} = \Lambda_{i+1,K},
$$
as required.  Therefore we have deduced by induction that
$$
\Lambda_{i+1}\sim \Lambda_{i+1,\ell_i/2} = \Lambda_{i}.
$$
Finally, if $i-1 \geq I$, we compute the coefficients of $\Lambda_i$, to complete the outer induction.

The coefficient of $e_{i-1}$ in $\Lambda_i$ is the same as the coefficient of $e_{i-1}$ in $M_0\underline{\lambda}_0+X_0$, which is $\lambda_{i-1}+E_{i-1}$, because we have added no roots involving $\epsilon_{i-1}$.

Since $i-1\geq I$ and $\ell_{i-1,\br}$ is even, we have $\lambda_i +1\leq \lambda_{i-1}+1\leq \frac12\ell_{i-1,\br} \leq \frac12\ell_{i,\br}$.  Therefore the coefficient of $e_{m+j}$, for all $\lambda_{i-1}+1\leq j \leq \frac12\ell_{i-1,\br}$ was $F_j-(m-i)$ in $\Lambda_{i+1}$ and is thus
$F_j-(m-i)-1$ in $\Lambda_i$, since we subtracted each $e_{m+j}$ exactly once in the course of computing $\Lambda_i$. The induction is complete.
\end{proof}

We may now prove our main theorem for very even Borel subalgebras.

\begin{theo}\label{T:veryeven}
Let $\br$ be a very even Borel subalgebra in $\mathcal{B}^d$.  Set $X_\br = M_0\rho_\br$ and let $M_\br \in \mathcal{C}$.  Let $\tau_\br$ be the affine function defined on $x\in \h^*$ by
$$
\tau_\br(x) = M_\br x + X_\br.
$$
Then for every $\lambda\in \mathscr{H}(m|n)$ we have $\tau_\br(\underline{\lambda}_\br)\sim \tau_0(\underline{\lambda}_0)$ so that 
\[P_{\mu,\frac{1}{2}}\circ\tau_\br\pr{\underline{\lambda}_{\br}}=P_{\mu,\frac{1}{2}}\circ\tau_{0}\pr{\underline{\lambda}_{0}}.\]
Consequently, the eigenvalue of the Capelli operator  $D^{\mu}$ on the irreducible component $W_\lambda$ is given by the evaluation of a polynomial in the highest weight $\underline{\lambda}_\br$ with respect to $\br$. 
\end{theo}

\begin{proof} 
It suffices to show  $\tau_\br(\underline{\lambda}_{\br}) \sim \tau_0(\underline{\lambda}_0)$ for all $\lambda \in \mathscr{H}(m|n)$. We have that 
\begin{align*}
M_\br\underline{\lambda}_{\br} + X_\br &= M_\br(\underline{\lambda}_0 - r(\lambda,\br))+X_\br\\
&= M_\br(\underline{\lambda}_0 + r_{\br} - r(\lambda,\br))+ (-M_\br r_{\br}+X_\br).
\end{align*}
Since $\underline{\lambda}_0 + r_{\br} - r(\lambda,\br)\in \mathfrak{a}^*$, we may replace $M_\br$ by $M_0$ in the first term, and as noted in Lemma~\ref{L:Xbr} the latter term is equal to $X_0$.  Therefore we have
$$
M_\br\underline{\lambda}_{\br} + X_\br = M_0(\underline{\lambda}_0 + r_{\br}-r(\lambda,\br))+X_0,
$$
and thus by applying Proposition~\ref{P:equiveven} we conclude that $M_\br\underline{\lambda}_{\br} + X_\br\sim M_0\underline{\lambda}_0+X_0$, which completes the proof.
\end{proof}

\subsection{A solution for generic weights} \label{SS:generic} %The set $\mathcal{C}_\br$}
From now on we suppose that $\br\in \mathcal{B}^d$ is not very even.  We first characterize the non-adjacent pairs $\{\delta_{2k}, \delta_{2k-1}\}$ in the $\delta\epsilon$ sequence corresponding to $\br$.

\begin{defn}
For each $1\leq k \leq n$, let 
\begin{align*}
T_\br&= \{ 1\leq k \leq n \mid j_{2k,\br}<j_{2k-1,\br} \}\\
&=  \{ 1\leq k \leq n\mid \exists 1\leq i \leq m : \ell_{i,\br}=2k-1\}.
\end{align*}
\end{defn}

One can infer the equality of the given sets from the combinatorial description of $\delta\epsilon$ sequences in Definition~\ref{D:ellj}.

\begin{lem}\label{def:be}
Define $\br_e\in \mathcal{B}^d$ by the values 
$$
\ell_{i,\br_e}= 2 \left\lfloor \frac12 \ell_{i,\br} \right\rfloor
$$
for all $1\leq i \leq m$.  Then $\br_e$ is a very even Borel subalgebra and the sequence of odd reflections $r_\alpha$ that take $\br_e$ to $\br$ corresponds to the set of roots $\{\delta_{2k-1}-\epsilon_i \mid k \in T_\br, \ell_{i,\br}=2k-1\}$.
%$ of the form $\alpha = \delta_{2k-1}-\epsilon_i$, with $k\in T_\br$, where the values $k$ may be repeated but the values  $i$ are distinct.
\end{lem}

\begin{proof}
Since $\br$ is a decreasing Borel subalgebra if and only if the sequence $(\ell_{i,\br})_{i=1}^m$ is increasing (and bounded by $2n$), it follows that the increasing sequence $(\ell_{i,\br_e})_{i=1}^m$ uniquely defines a decreasing Borel subalgebra $\br_e$, and that it is very even.  
In the $\delta\epsilon$ sequence of $\br_e$, the terms $\delta_{2k-1}\delta_{2k}$ are adjacent.  For each $k\in T_\br$ and $i$ such that $\ell_{i,\br}=2k-1$, we have $\ell_{i,\br_e}=2k-2$, but for all other $i$, $\ell_{i,\br}$ must be even and therefore $\ell_{i,\br}=\ell_{i,\br_e}$.  Thus to obtain the $\delta\epsilon$ sequence of $\br$ one must, for each $k\in T_\br$, shift $\delta_{2k-1}$ to the right using the reflections corresponding to $\{\delta_{2k-1}-\epsilon_i \mid  \ell_{i,\br}=2k-1\}$.
\end{proof}

Consequently, if we define for each $k\in T_\br$
\begin{equation}\label{eq:rtakingbetob}
r_{odd,\br,k} := (j_{2k-1}-j_{2k}) \delta_{2k-1} - \sum_{i= m-j_{2k-1}+1}^{m-j_{2k}}\epsilon_i,
\end{equation}
then we have
\begin{equation}\label{eq:rbrassum}
r_\br = r_{\br_e} + \sum_{k\in T_\br} r_{odd,\br,k}.
\end{equation}

\begin{defn}\label{D:Cbfirst}
Let $\br\in \mathcal{B}^d$.  Set
$$
\mathcal{C}_\br = \{ M\in \mathcal{C} \mid \forall k\in T_\br, \;Mr_{odd,\br,k}=0\}.
$$
\end{defn}

\begin{lem}\label{lem:nonempty}
The set $\mathcal{C}_\br$ is nonempty. 
\end{lem}

\begin{proof}
For any $M = M_0 + A' \in \mathcal{C}$ where $A'=[0_{(m+n)\times m}| A]$ for some $A$ with columns satisfying $A_{2k}=-A_{2k-1}$ for all $k$, one has
$$
Mr_{odd,\br,k} = M_0r_{odd,\br,k} + A(j_{2k-1}-j_{2k}) \delta_{2k-1};
$$
thus to satisfy the condition $Mr_{odd,\br,k}=0$, one must set
$$
A_{2k-1} := \frac{-1}{j_{2k-1}-j_{2k}}M_0r_{odd,\br,k}
$$
and $A_{2k}:=-A_{2k-1}.$
These conditions are distinct as $k$ runs over $T_\br$. We conclude that $\mathcal{C}_\br$ is nonempty.
\end{proof}

The set $\mathcal{C}_\br$  contains a unique element if and only if $T_\br = \{1, 2, \cdots, n\}$.

\begin{prop}\label{prop:compatiblegeneral}
Set $X_\br= X_{\br_e}$ and suppose $\lambda\in \mathscr{H}(m|n)$ is $\br$-generic.  Then for all $M_\br\in \mathcal{C}_{\br}$ we have
$$
M_\br \underline{\lambda}_{\br} + X_\br= M_0 \underline{\lambda}_0 + X_0.
$$
\end{prop}

\begin{proof}
Suppose $\lambda$ is $\br$-generic and suppose $M\in \mathcal{C}_\br$.  Then since 
$$
\underline{\lambda}_\br = \underline{\lambda}_0 - r_\br = \underline{\lambda}_{\br_e}+ \sum_{k\in T_\br}r_{odd,\br,k},
$$
we have, by the conditions defining $\mathcal{C}_\br$ and the choice of $X_\br$, that
$$
M\underline{\lambda}_\br+X_\br = M\underline{\lambda}_{\br_e}+X_{\br_e}.
$$
Moreover, if $\lambda$ is $\br$-generic, then it is also $\br_e$-generic since  $\ell_{i,\br_e}\leq \ell_{i,\br}$ for all $1\leq i\leq m$.   Therefore by Theorem~\ref{T:veryeven} and Proposition~\ref{Prop:equalgeneric}, we infer that 
$$
M \underline{\lambda}_{\br_e} + X_{\br_e} = M_0\underline{\lambda}_0 + X_0,
$$
as required.
\end{proof}

\subsection{A solution for the relatively even case}\label{SS:relativelyeven}
%Compatibility of $\lambda$ and $\br$}

Suppose now that $\lambda$ is not $\br$-generic, but that for each $k\in T_\br$, the set $\mathcal{N}_\br$ of roots defining  $r(\lambda,\br)$ include either all, or none, of the roots defining $r_{odd,\br,k}$.  We will call such $\lambda$ \emph{compatible} with $\br$, but to define compatibility precisely, we must first make some observations about $\delta\epsilon$ sequences.

Suppose that in the $\delta\epsilon$ sequence for $\br$, there exists at least one $\epsilon$ between $\delta_{2k}$ and $\delta_{2k-1}$.  
  Since $j_{2k-1}$ is the number of $\epsilon$ to the left of $\delta_{2k-1}$, and these are in decreasing order,  we deduce that these terms are precisely $\epsilon_m, \cdots, \epsilon_{m-j_{2k-1}+1}$.  By the same reasoning, $\epsilon_{m-j_{2k}}$ is the first $\epsilon$ to the right of $\delta_{2k}$.  Therefore, the $\delta\epsilon$  sequence contains
\begin{equation}\label{eq:deltaepsilonexcerpt}
  \cdots \delta_{2k}\epsilon_{m-j_{2k}}\cdots \epsilon_{m-j_{2k-1}+1}\delta_{2k-1}\cdots.
\end{equation}

Now suppose that $\lambda$ is not $\br$-generic.  Since the sequence of $\lambda_i$ is decreasing, we have that
\begin{equation}\label{eq:interval}
\lambda_{m-j_{2k}} \leq \lambda_{m-j_{2k}-1} \leq  \cdots \leq \lambda_{m-j_{2k-1}+1}.
\end{equation}
On the other hand, the value of $\ell_i$, for each $m-j_{2k} \leq i \leq m-j_{2k-1}+1$, is exactly $2k-1$. Therefore, depending on the relation among $2k-1,2\lambda_{m-j_{2k}}$ and $2\lambda_{m-j_{2k-1}+1}$, we should subtract all, none or only some of the roots from the set \[\{\delta_{2k-1}-\epsilon_{m-j_{2k}},\ldots,\delta_{2k-1}-\epsilon_{m-j_{2k-1}+1}\}\]
to obtain $\underline{\lambda}_\br$ from $\underline{\lambda}_0$.

These observations motivate the following definition.

\begin{defn}
We say that $\lambda \in \mathscr{H}(m|n)$ 
 is \emph{compatible} with $\br$ if it is either $\br$-generic or, for each $k\in T_\br$, we have either $2\lambda_{m-j_{2k}} > 2k-1$  or $2\lambda_{m-j_{2k-1}+1} < 2k-1$. 
\end{defn}

\begin{prop}\label{P:compatibleMN}
Let $\br \in \mathcal{B}^d$ and set $X_\br = X_{\br_e}$. Then for any $M\in \mathcal{C}_\br$, the affine function
$$
\tau_\br(x) = M_\br x + X_\br
$$
satisfies $\tau_\br(\underline{\lambda}_\br)\sim \tau_0(\underline{\lambda}_0)$ for every  $\lambda \in \mathscr{H}(m|n)$ that is compatible with $\br$.
\end{prop}

For ease of reference in the proof, for each $\lambda$ that is compatible with $\br$, define
\begin{equation}\label{eq:Tbl}
T_{\br,\lambda} \coloneqq\{k\in T_\br \mid 2\lambda_{m-j_{2k,\br}} > 2k-1 \}.
\end{equation}

\begin{proof}[Proof of  Proposition~\ref{P:compatibleMN}]
Suppose $\lambda$ is compatible with $\br$. If it is in fact $\br$-generic then Proposition~\ref{prop:compatiblegeneral} implies that $\tau_\br(\underline{\lambda}_\br)=\tau_0(\underline{\lambda}_0)$, so we are done. 

Otherwise, recall that the set of roots $\mathcal{G}_\br\setminus \mathcal{N}_\br$ determining $r_{\br}-r(\lambda,\br)$ is given by \eqref{E:doublestar}, 
where $I=I_{\lambda,\br}$ denotes the least index for which $\ell_{I,\br}>2\lambda_I$.
Recall also that for each $k\in T_\br$, $r_{odd,\br,k} = (j_{2k-1}-j_{2k}) \delta_{2k-1} - \sum_{i= m-j_{2k-1}+1}^{m-j_{2k}}\epsilon_i$, and that for each index $i$ in this sum, $\ell_i = 2k-1$.
The condition of compatibility implies that either
$$
\ell_i = 2k-1 < 2\lambda_{m-j_{2k}} \leq 2\lambda_{m-j_{2k}-1} \leq  \cdots \leq 2\lambda_{m-j_{2k-1}+1}
$$
in which case, $I > i$ and none of these roots appear in the expression \eqref{E:doublestar}, or
$$
2\lambda_{m-j_{2k}} \leq 2\lambda_{m-j_{2k}-1} \leq  \cdots \leq 2\lambda_{m-j_{2k-1}+1} < 2k-1 = \ell_i
$$
in which case $I\leq i$ and every root vector that occurs in  $r_{odd,\br,k}$ appears in \eqref{E:doublestar}.

Therefore, the set $T_{\br,\lambda}$ is exactly the set of $k\in T_\br$ such that none of the roots determining $r_{odd,\br,k}$ occur among the roots in \eqref{E:doublestar}; this set may be empty. 

Therefore by construction we have
$$
r(\lambda,\br) = r(\lambda,\br_e) + \sum_{k\in T_{b,\lambda}}r_{odd,\br,k},
$$
and so we deduce
\begin{equation}\label{eq:relationbrbre}
\underline{\lambda}_{\br} = \underline{\lambda}_0 - r(\lambda,\br) = \underline{\lambda}_0 - r(\lambda,\br_e)-\sum_{k\in T_{\br,\lambda}}r_{odd,\br,k} 
= \underline{\lambda}_{\br_e} -\sum_{k\in T_{\br,\lambda}}r_{odd,\br,k}.
\end{equation}
Let $M_\br\in\mathcal{C}_{\br}$; since $T_{\br,\lambda}\subset T_{\br}$, this implies
$$
M_\br\underline{\lambda}_{\br}  = M_\br(\underline{\lambda}_{\br_e}) - \sum_{k\in T_{b,\lambda}}M_\br r_{odd,\br,k} = M_\br(\underline{\lambda}_{\br_e}).
$$
Thus since $X_\br = X_{\br_e}$, we have that $\tau_\br(\underline{\lambda}_\br) = \tau_{\br_e}(\underline{\lambda}_{\br_e})$.  Finally, as $\br_e$ is very even and $M_\br\in \mathcal{C}$, we can  apply Theorem~\ref{T:veryeven} to deduce that
$$
M_\br\underline{\lambda}_{\br_e} + X_{\br_e} \sim M_0\underline{\lambda}_0+X_0,
$$
and the proposition follows.
\end{proof}

We conclude this section by defining the class of Borel subalgebras $\br$ for which all $\lambda\in \mathscr{H}(m|n)$ are compatible with $\br$.

\begin{defn} \label{D:relativelyeven}
Let $\br \in \mathcal{B}^d$.  Then $\br$, or its corresponding $\delta\epsilon$ sequence, is called \emph{relatively even} if for every $1\leq k\leq n$ we have $j_{2k-1}-j_{2k}\leq 1$, that is, 
 if there is at most one $\epsilon$ between 
 $\delta_{{2k}}$ and  $\delta_{{2k-1}}$.
 \end{defn} 

\begin{cor} \label{prop:reiscomp}
Suppose $\br\in \mathcal{B}^d$ is relatively even.  Then for any $M_\br\in \mathcal{C}_\br$, the affine function $\tau_\br(x) = M_\br x + X_{\br_e}$
satisfies 
\[P_{\mu,\frac{1}{2}}\circ\tau_\br\pr{\underline{\lambda}_{\br}}=P_{\mu,\frac{1}{2}}\circ\tau_{0}\pr{\underline{\lambda}_{0}}\]
for every $\lambda \in \mathscr{H}(m|n)$.
Consequently, the eigenvalue of the Capelli operator  $D^{\mu}$ on the irreducible component $W_\lambda$ is given by the evaluation of a polynomial in the highest weight $\underline{\lambda}_\br$ with respect to $\br$.
\end{cor}

\begin{proof}
Let $\lambda \in \mathscr{H}(m|n)$.  If $\lambda$ is not compatible with $\br$, then
there exists an index $k\in T_\br$ for which 
\[
2\lambda_{m-{j_{2k}}} < 2k-1 < 2\lambda_{m-j_{2k-1}+1} .
\]
This is impossible if $j_{2k}=j_{2k-1}$ (as the $\lambda_i$ are decreasing) and if $j_{2k-1}-j_{2k}=1$ (as then $m-j_{2k}=m-j_{2k-1}+1$).  Therefore, if $\br$ is relatively even it is compatible with every $\lambda \in \mathscr{H}(m|n)$. 
That the affine function $\tau_\br$ gives a solution to the refined CEP thus follows from  Proposition~\ref{P:compatibleMN}.
\end{proof}

\subsection{A different approach for $\br$ that are not relatively even}
\label{ss:full}
Unfortunately, following the approach of the preceding section to establish a formula for $\tau_\br$ when $\br$ is not relatively even leads to a piecewise affine function, as in \cite{MengyuanPhD}.  In this section, we follow a different approach.

Recall that the functions in $\Lambda_{m,n,\theta}$ are constant on the orbits $\mathcal{O}$ of elements $\C^{m+n}$ under monoidal symmetry.  Since they are polynomial functions, they are constant on the Zariski closure of $\mathcal{O}$.  As $\mathcal{O}$ is a discrete set, if it is infinite, then its Zariski closure must contain (infinitely) many orbits, and in particular,  %  Thus, whenever $\mathcal{O}$ is infinite, 
the polynomials in $\Lambda_{m,n,\theta}$ will not separate orbits.  That is, we may have $P(x)=P(y)$ for $x,y$ in distinct orbits, for all $P\in \Lambda_{m,n,\theta}$.

The question of when infinite orbits exist was answered in  \cite{SVorbits}, and in particular this occurs for  $\theta=1/2$.  The following is a consequence of their results.

\begin{lem}\label{lem:orbits}
Suppose $x\in \C^{m|n}$ and that $1\leq i < i_0\leq m$, $1\leq j\leq n$.  Let $\mathrm{orb}(x)$ denote the orbit of $x$ under monoidal symmetry. If $x_i-x_{i_0} =  2x_i + x_{m+j} = -\frac12$ then 
\begin{equation}\label{eq:orbitclosure}
\overline{\mathrm{orb}(x)} \supset \{ u\in \C^{m|n} \mid u_i-u_{i_0} = 2u_i+u_{m+j} = \pm \frac12, 
u_s=x_s \forall s\notin\{i,i_0,m+j\}\},
\end{equation}
where $\overline{\mathrm{orb}(x)}$ denotes the Zariski closure of the orbit of $x$.
\end{lem}

\begin{proof}
In \cite[\S5 (22)]{SVorbits}, the authors provide a key example for with $\theta=1/2$.  Namely, for any $a\in \C$, it is straightforward to show that 
\begin{align*}
\mathcal{O}(a) = \{&(\ell+a-\frac12,\ell+a,-2(\ell+a)+\frac12),
 (\ell+a+\frac12, \ell+a, -2(\ell+a)-\frac12), \\
 &(\ell+a,\ell+a-\frac12, -2(\ell+a)+\frac12), (\ell+a,\ell+a+\frac12, -2(\ell+a)-\frac12) \mid \ell \in \Z\}
\end{align*}
is an orbit in $\C^{2|1}$ under monoidal symmetry.     The Zariski closure $\overline{\mathcal{O}(a)}$ is thus equal to  $\cup_{a\in \C} \mathcal{O}(a)$, which is the union of the two affine lines:
$$
u_1-u_2=2u_1+u_3=\pm \frac12, \quad u_2-u_1=2u_2+u_3=\pm\frac12.
$$
Applying this to the triple $(i,i_0|j)$ of indices in place of $(1,2|3)$ %, and specializing to one particular line, 
gives the result.
%Consequently, for any  $1\leq i, i_0\leq m$, $1\leq j \leq n$, if the coefficients of  $e_i, e_{i_0}, e_{m+j}$ in $x=\sum_ix_ie_i\in \C^{m|n}$ lie in some $\mathcal{O}(a)$, then  
%\begin{equation}\label{eq:orbitclosure}
%\overline{\mathrm{orb}(x)} \supset \{ u\in \C^{m|n} \mid u_i-u_{i_0} = 2u_i+u_{m+j} = \pm \frac12, 
%u_s=x_s \forall s\notin\{i,i_0,m+j\}\},
%\end{equation}
%where $\overline{\mathrm{orb}(x)}$ denotes the Zariski closure of the orbit of $x$. 
\end{proof}

We now proceed to the definition of our affine function, which we will denote $\tau_\br^{\mathrm{full}}$ to distinguish it from the function $\tau_\br$ defined on relatively even Borel subalgebras in the previous section.  

\begin{defn}\label{D:Cbfull}
Let $\br \in \mathcal{B}^d$.  Define
$$
\mathcal{C}_\br^{\mathrm{full}} = \{ M \in \mathcal{C} \mid \forall k\in T_\br, M\delta_{2k-1}= \frac12e_{m-j_{2k,\br}}-e_{m+k}\}.
$$
Given any $M_\br\in\mathcal C^{\mathrm{full}}_\br$, we set  $X_\br = M_\br r_\br + X_0$ and 
$\tau_\br^{\mathrm{full}}(x) = M_\br x+X_\br$ for  $x\in \mathfrak{h}^*$.
\end{defn}

Note that the defining condition of $\mathcal{C}_\br^{\mathrm{full}}$
basically describes the columns $A_{2k-1}$ of $M_\br$ for $k\in T_\br$, hence $\mathcal{C}_\br^{\mathrm{full}}$ is nonempty.

\begin{remark}
Indeed the value of $M_\br r_\br$ is independent of the choice of $M_\br\in\mathcal C^{\mathrm{full}}_\br$. This can be seen as follows.
By \eqref{eq:rbrassum} we have $r_\br= r_{\br_e}+\sum_{k\in T_\br}r_{odd,\br,k}$.  Since $M_\br \in \mathcal C$, the value of $M_\br r_{\br_e}$ is independent of the choice of $M_\br$. Also, by the definition of $M_\br$ the value of $M_\br r_{odd,\br, k}$ is independent of the choice of $M_\br\in\mathcal C^{\mathrm{full}}_\br$. 
\end{remark}

\begin{theo}\label{T:fullcase}
For every $\lambda\in \mathscr{H}(m|n)$ we have  
\[
P_{\mu,\frac{1}{2}}\circ\tau_\br^{\mathrm{full}}\pr{\underline{\lambda}_{\br}}=P_{\mu,\frac{1}{2}}\circ\tau_{0}\pr{\underline{\lambda}_{0}}.
\]
Consequently, the eigenvalue of the Capelli operator  $D^{\mu}$ on the irreducible component $W_\lambda$ is given by the evaluation of a polynomial in the highest weight $\underline{\lambda}_\br$ with respect to $\br$. 
\end{theo}

To prove this, we require several lemmas.
As in Lemma~\ref{def:be}, let $\br_e$ be the closest very even Borel subalgebra to $\br$; we will relate $\tau_\br^{\mathrm{full}}$ to the function $\tau_{\br_e}$ that was defined in Theorem~\ref{T:veryeven}.  For ease of notation, set $\ell_i = \ell_{i,\br}$ and $j_{t}=j_{t,\br}$ in what follows.   

For each $k\in T_\br$, the $\delta\epsilon$ sequence for $\br$ contains
$$
  \cdots \delta_{2k}\epsilon_{m-j_{2k}}\cdots \epsilon_{m-j_{2k-1}+1}\delta_{2k-1}\cdots
$$ % subsequence \eqref{eq:deltaepsilonexcerpt} 
whereas that for $\br_e$ contains
\begin{equation}\label{eq:deltaepeven}
\cdots \delta_{2k}\delta_{2k-1}\epsilon_{m-j_{2k}}\cdots \epsilon_{m-j_{2k-1}+1}\cdots.
\end{equation}
Thus $\ell_{i} = 2k-1$ and $\ell_{i,\br_e} = 2k-2$ for $m-j_{2k-1}+1\leq i \leq m-j_{2k}$; similarly, $j_{2k-1,\br_e}=j_{2k,\br_e}$ whereas by hypothesis $j_{2k-1}-j_{2k}\geq 1$. 

We have 
$$
r_\br = r_{\br_e} + \sum_{k\in T_\br} r_{odd,\br,k}
$$
as in \eqref{eq:rbrassum}, where 
$$
r_{odd,\br,k} = (\delta_{2k-1} - \epsilon_{m-j_{2k}}) + (\delta_{2k-1} - \epsilon_{m-j_{2k}-1}) + \cdots + (\delta_{2k-1} - \epsilon_{m-j_{2k-1}+1}).
$$
To obtain $\br$ from $\br_e$, we must apply the reflections in the roots $\delta_{2k-1}-\epsilon_i$ in decreasing lexicographic order, that is, for $k$ decreasing in $T_\br$ and then $i$ decreasing from $m-j_{2k}$ until $m-j_{2k-1}+1$.  

\begin{defn}
Write $(k,i)\succeq (k',i')$ if we have $k,k'\in T_\br$, $m-j_{2k-1}+1\leq i \leq m-j_{2k}$, $m-j_{2k'-1}+1\leq i' \leq m-j_{2k'}$ and $k\geq k'$, $i\geq i'$.  Write $(k,i)\succ (k',i')$ if $(k,i)\succeq (k',i')$ and $(k,i)\neq (k',i')$.
\end{defn}

Note that we retain the index $k$ in the notation only for convenience. For two such pairs, the statement $(k,i)\succ(k',i')$ is equivalent to stating $i>i'$, as in each pair the first term is fully determined by the second.  
In particular, we may write $(k,i)\succ(k',i')$  for any $1\leq i'\leq m$, not only those corresponding to $k'\in T_\br$ and $m-j_{2k'+1}\leq i' \leq m-j_{2k'}$, without ambiguity.  
We require two technical lemmas.

\begin{lem} \label{lem:criticalcase}
Fix $M_\br\in\mathcal C_\br^{\mathrm{full}}$ and let $k\in T_\br$.  Suppose the coefficient of $\epsilon_{m-j_{2k}}$ in $\underline{\lambda}_{\br_e}$ is zero.  Then we have
\begin{equation}\label{eq:lem1}
\tau_{\br_e}(\underline{\lambda}_{\br_e}) + M_\br \sum_{k'>k} r_{odd,\br,k'}  \sim 
\tau_{\br_e}(\underline{\lambda}_{\br_e}) + M_\br \sum_{k'>k} r_{odd,\br,k'} + (e_{m-j_{2k}}-e_{m+k})
\end{equation}
where $\sim$ denotes monoidal symmetry.
\end{lem}

\begin{proof}
Recall that  $\tau_{\br_e}(x)=
M_{\br_e}x +X_{\br_e}=
M_0x+M_0r_{\br_e}+X_0$.
As in the proof of Proposition~\ref{P:equiveven}, by a direct calculation using~\eqref{E:doublestar}  
we have \[
M_0(r_{\br_e}-r(\lambda,\br_e))=
\sum_{i=I_{\lambda,\br_e}}^{m}\sum_{j=\lambda_i+1}^{\frac12\ell_{i,\br_e}}(e_i-e_{m+j}).
\]
Hence
\begin{equation}\label{eq:summingtombe}
\tau_{\br_e}(\underline{\lambda}_{\br_e})=M_{\br_e}\underline{\lambda}_{\br_e} +X_{\br_e} =
M_0\underline{\lambda}_{0}+X_0 + \sum_{i=I_{\lambda,\br_e}}^{m}\sum_{j=\lambda_i+1}^{\frac12\ell_{i,\br_e}}(e_i-e_{m+j}).
\end{equation}
The hypothesis implies $2\lambda_{m-j_{2k}}\leq \ell_{m-j_{2k},\br_e}=2k-2$. 
Thus
for $i>m-j_{2k}$ we have
\[
2\lambda_i\leq 2\lambda_{m-j_{2k}}
\leq \ell_{m-j_{2k},\br_e}\leq \ell_{i,\br_e}-2.
\]
This implies that $I_{\lambda,\br_e}\leq m-j_{2k}+1$.  

Recall the formula for the sums $r_{odd,\br,k}$ defined in~\eqref{eq:rtakingbetob}.
First note that the coefficients of $e_{m-j_{2k}}$ and $e_{m+k}$ in 
$M_\br \sum_{k'>k} r_{odd,\br,k'}$ are zero.  Namely, by the definition of $M_\br$ we have that $M_\br(\epsilon_i) = -\frac12e_i$ and $M_\br(\delta_{2k'-1}) = \frac12 e_{m-j_{2k'}}-e_{m+k'}$.  The only pairs $(k',i)$ occuring in this sum have $k'>k$ and $i>m-j_{2k}$.

Next we  compute the coefficients of 
$e_{m-j_{2k}}$ and $e_{m+k}$ in
$\tau_{\br_e}(\underline{\lambda}_{\br_e}) = \sum_i x_ie_i$.

Let us begin with  $x_{m-j_{2k}}$.   The coefficient of $e_{m-j_{2k}}$ in $M_0\underline{\lambda}_{0}$ is $\lambda_{m-j_{2k}}$ and in $X_0$ is $E_{m-j_{2k}}$.  
 If it happens that $I_{\lambda,\br_e}\leq m-j_{2k}$, then the index $i=m-j_{2k}$ occurs in the sum \eqref{eq:summingtombe} whence the coefficient of $e_{m-j_{2k}}$ in this sum is $\frac12\ell_{m-j_{2k},\br_e}-\lambda_{m-j_{2k}}$.  Since
$\frac12\ell_{m-j_{2k},\br_e}=k-1$, we deduce that $x_{m-j_{2k}}=E_{m-j_{2k}}+k-1$. If indeed $I_{\lambda,\br_e}=m-j_{2k}+1$, then we must have 
$2\lambda_{m-j_{2k}}=2k-2$ and the index $i=m-j_{2k}$ does not occur in~\eqref{eq:summingtombe}, hence again we have \[
x_{m-j_{2k}}=\lambda_{m-j_{2k}}+E_{m-j_{2k}}=E_{m-j_{2k}}+k-1.
\]
Next we compute $x_{m+k}$.  Since $2\lambda_{m} \leq 2\lambda_{m-j_{2k}} \leq \ell_{m-j_{2k},\br_e}=2k-2$, we deduce from the structure of the hook partition that $\mu_{2k-1}=\mu_{2k}=0$, so the coefficient of $e_{m+k}$ in $M_0\underline{\lambda}_0$ is zero.   The coefficient of $e_{m+k}$ in $X_0$ is $F_k$.  The coefficient of $e_{m+k}$ in the sum \eqref{eq:summingtombe} is the number of indices $i$  such that 
$\lambda_i<k\leq \ell_{i,\br_e}/2$.  The inequality 
$k\leq \ell_{i,\br_e}$ implies $\epsilon_i$ is to the left of $\delta_{2k}$ in $\br_e$ in \eqref{eq:deltaepeven} and in particular $i>m-j_{2k}$; but then $\lambda_i\leq \lambda_{m-j_{2k}}\leq k-1$ holds.  Thus the coefficient is $-(m-(m-j_{2k}))=-j_{2k}$, whence $x_{m+k} = F_k-j_{2k}$.

Now note that $E_{m-i}+\frac12F_j=\frac34+\frac12i-j$, so that
$$
x_{m-j_{2k}}+\frac12 x_{m+k} = \frac34+\frac12j_{2k} - k + k-1 - j_{2k}=-\frac14,
$$
which implies that monoidal symmetry may be applied, yielding the result.
\end{proof}

\begin{lem} \label{lem:othercases}
Fix $M_\br\in\mathcal C_\br^{\mathrm{full}}$.  Let $k\in T_\br$ and suppose $m-j_{2k-1}+1\leq i <  m-j_{2k}$.  Suppose the coefficient of $\epsilon_{i}$ in $\underline{\lambda}_{\br_e}$ is zero.  Then we have
\begin{align}\label{eq:lem2}
\tau_{\br_e}(\underline{\lambda}_{\br_e}) &+ M_\br \sum_{(k',i')\succ(k,i)} (\delta_{2k'-1}-\epsilon_{i'}) \\ 
\notag &  \sim_E \quad
\tau_{\br_e}(\underline{\lambda}_{\br_e}) + M_\br \sum_{(k',i')\succ(k,i)} (\delta_{2k'-1}-\epsilon_{i'})  + \left(\frac12e_{m-j_{2k}}+\frac12e_i-e_{m+k}\right)
\end{align}
although they are not related by monoidal symmetry.
\end{lem}

\begin{proof}
We set $i_0 = m-j_{2k}$ and show that with $x=\sum_t x_te_t$ equal to the left hand side above, we may apply Lemma~\ref{lem:orbits} to the triple $(i,i_0,k)$.  Since $\frac12e_{m-j_{2k}}+\frac12e_i-e_{m+k}$ has coefficients $(\frac12, \frac12, -1)$, it then follows that the left side of the relation lies in $\overline{\mathrm{orb}(x)}$ as well, as required.

As noted in the previous proof, the coefficients of $\{e_i,e_{i_0}, e_{m+k}\}$ in $x$ coincide with those of 
\begin{equation}\label{eq:intermediate}
\tau_{\br_e}(\underline{\lambda}_{\br_e}) + M_\br \sum_{i<i'\leq m-j_{2k}} (\delta_{2k-1}-\epsilon_{i'}) 
\end{equation}
as the terms with $k'>k$ cannot contribute.  

First we compute the coefficients of 
$\{e_i,e_{i_0},e_{m+k}\}$ in 
the right hand side of~\eqref{eq:summingtombe}.
Since the coefficient of $\epsilon_i$ in $\underline{\lambda}_{\br_e}$ is zero, we have $\lambda_{m-j_{2k}}\leq \lambda_i \leq \ell_{i,\br_e}/2 = k-1$.
It follows that for $i'>m-j_{2k}$ we have \[
\lambda_{i'}\leq \lambda_{m-j_{2k}}\leq k-1<\ell_{i',\br_e}/2
\] and  
in particular, $I_{\lambda,\br_e}\leq  m-j_{2k}+1$.

Now consider any $s\leq m$. The coefficient of $e_{s}$ in $M_0\lambda_0+X_0$ is $\lambda_s+E_s$.   If
$s\geq I_{\lambda,\br_e}$, then the coefficient of $e_s$ in the sum on the right hand side of \eqref{eq:summingtombe} is $\frac12 \ell_{s,\br_e}-\lambda_s = k-1-\lambda_s$. Hence the coefficient of $e-s$ on the right hand side of~\eqref{eq:summingtombe} is 
$\lambda_s+E_s+(k-1-\lambda_s)=E_s+k-1$. 
If $i\leq s<I_{\lambda,\br_e}$, then  $e_s$ does not occur in
the sum on the right hand side of \eqref{eq:summingtombe} but we have 
\[
k-1=\ell_{i,\br_e}\leq 
\ell_{s,\br_e}\leq \lambda_s\leq \lambda_i\leq k-1,
\]
from which it follows that $\lambda_s=k-1$. This means that in this case also the coefficient of $e_s$ on the right hand side of~\eqref{eq:summingtombe}
is equal to $\lambda_s+E_s=E_s+k-1$. As a consequence, we obtain that the coefficient of $e_s$ for $s=i,i_0$
in~\eqref{eq:summingtombe} is $E_s+k-1$.

On the other hand, from $\lambda_m\leq k-1$ it follows that $\mu_{2k-1}=\mu_{2k}=0$ so the coefficient of $e_{m+k}$ in $M_0\lambda_0+X_0$  is $F_k$.  As in the  proof of Lemma~\ref{lem:criticalcase}, the coefficient of $e_{m+k}$ in the sum on the right hand side of \eqref{eq:summingtombe} is $-j_{2k}$.

We now compute the contributions of the second term in \eqref{eq:intermediate}. First note that
$$
 \sum_{i<i'\leq m-j_{2k}} (\delta_{2k-1}-\epsilon_{i'})
= (m-j_{2k}-i)\delta_{2k-1} - \sum_{i'=i+1}^{m-j_{2k}} \epsilon_{i'}.
$$
Therefore from the definition of $M_\br$ we have
\begin{align*}
M_\br \sum_{i<i'\leq m-j_{2k}} (\delta_{2k-1}-\epsilon_{i'}) &=
 (m-j_{2k}-i)M_\br\delta_{2k-1} - \sum_{i'=i+1}^{m-j_{2k}} M_\br\epsilon_{i'}\\
 &= (m-j_{2k}-i)(\frac12e_{m-j_{2k}}-e_{m+k}) - \sum_{i'=i+1}^{m-j_{2k}}(-\frac12 e_{i'}).
 \end{align*}
Therefore, altogether we have
\begin{align*}
x_{i} &= E_i +k-1\\
x_{m-j_{2k}} &= E_{m-j_{2k}} + k-1 + \frac12(m-j_{2k}-i)+
\frac12\\
x_{m+k} &= F_k-j_{2k} -  (m-j_{2k}-i) = F_k-m+i.
\end{align*}
We compute
\begin{align*}
x_i-x_{m-j_{2k}} &= E_i - E_{m-j_{2k}} - \frac12(m-j_{2k}-i+1)\\
&= \frac14(-2i+2(m-j_{2k}))-\frac12(m-j_{2k}-i+1) = -\frac12
\end{align*}
and
\begin{align*}
2x_i+x_{m+k} &= 2(E_i+k-1)+(F_k-m+i)\\
&= \frac12(m-2n+1-2i)+2k-2 + \frac12(m+2n+2-4k) -m+i\\
&= \frac12-2+1 = -\frac12.
\end{align*}
It follows from Lemma~\ref{lem:orbits} that the orbit of $x$ is infinite and that $x+\frac12e_{m-j_{2k}}+\frac12e_i-e_{m+k}$ lies in its Zariski closure.
\end{proof}

\begin{proof}[Proof of Theorem~\ref{T:fullcase}]
We begin by relating $\tau_\br^{\mathrm{full}}$ to $\tau_{\br_e}$.  Let $I=I_{\lambda,\br}$ be the least index $i$ for which $2\lambda_i<\ell_{i,\br}$.  Then $2\lambda_I \leq \ell_{I,\br_e}$ and so by Remark~\ref{Rem:coeffszero}, for every $j\geq I$, the coefficient of $\epsilon_j$ in $\underline{\lambda}_{\br_e}$ is zero.

Set $k_\lambda = \lceil (1+\ell_{I,\br})\rceil/2$;  if $\ell_{I,\br}$ is odd, this is the index such that $m-j_{2k_\lambda}+1\leq I \leq m-j_{2k_\lambda}$, and otherwise it need not be a relevant index.  %even, then let $k_\lambda= \ell_{I,\br}/2$; otherwise set $k_\lambda=(\ell_{I,\br}+1)/2$.  Then
Then
$$
r(\lambda,\br)=r(\lambda,\br_e)  + \sum_{(k,i)\prec(k_\lambda,I)} (\delta_{2k-1}-\epsilon_i).
$$
%where this represents the sum over all pairs $(k,i)$ with $k\in T_\br$ and $m-j_{2k-1}+1\leq i \leq m-j_{2k}$, such that $(k,i)$ precedes $(k_\lambda,I)$ in lexicographic order.   
Since $M_\br \in \mathcal{C}$ and $\underline{\lambda}_\br =  \underline{\lambda}_{\br_e} - (r(\lambda,\br)-r(\lambda,\br_e))$, for any  $\lambda$ we have
\begin{align*}
M_\br\underline{\lambda}_\br+X_{\br} &=
M_\br(\underline{\lambda}_{\br_e} - (r(\lambda,\br)-r(\lambda,\br_e))) + (M_\br r_\br +X_0)\\
&= M_\br\underline{\lambda}_{\br_e} + X_{\br_e} + M_\br((r_\br-r(\lambda,\br))-(r_{\br_e}-r(\lambda,\br_e)))\\
&= \tau_{\br_e}(\underline{\lambda}_{\br_e}) + M_\br \sum_{(k,i)\succeq(k_\lambda,I)} (\delta_{2k-1}-\epsilon_i ).
\end{align*}
%{\color{magenta}\textbf{Hadi:} 
%In $\br$ we have
%\[
%\delta_{2k}\cdots \epsilon_i\cdots \delta_{2k-1}.
%\quad\text{and}\quad
%\delta_{\ell_I+1}\cdots \epsilon_I\cdots \delta_{\ell_I}.
%\]
%Since $i\geq I$, we also have $2k-1\geq \ell_{I,\br}$ or equivalently, $k\geq \lceil (\ell_{I,\br}+1)/2\rceil$. Thus, I think we should set $k_\lambda:=\lceil (\ell_{I,\br}+1)/2\rceil$.  \textcolor{blue}{But if it's $\epsilon_i\delta_{2k}\delta_{2k-1}$ we want $k$ not $k+1$ --- so how about $k_\lambda:=\lfloor (\ell_{I,\br}+1)/2\rfloor$?} }

%Again, this last sum is over all pairs $(k,i)\geq (k_{\lambda},I(\lambda,\br))$ in lexicographic order, such that $k\in T_\br$ and $m-j_{2k-1}+1\leq i\leq m-j_{2k}$.

  We now proceed by induction from $J=m+1$ down to $J=I$, showing that 
%We now show by induction on pairs $(K_J,J)$ (where 
(with $K_J=\lceil (1+\ell_{J,\br})\rceil/2$)  if
\begin{equation}\label{eq:inductionhyp}
\tau_{\br_e}(\underline{\lambda}_{\br_e}) + M_\br \sum_{(k,i)\succeq (K_J,J) }(\delta_{2k-1}-\epsilon_i)  \sim_E \tau_{\br_e}(\underline{\lambda}_{\br_e})
\end{equation}
for all $\lambda$ such that the coefficient of $\epsilon_J$ in $\underline{\lambda}_{\br_e}$ is $0$, then this same relation  also holds when we replace $J$ with $J-1$ and apply it to all $\lambda$ such that 
the coefficient of $\epsilon_{J-1}$ in 
$\underline{\lambda}_{\br_e}$ is $0$. Note that if the coefficient of $\epsilon_i$ in $\overline{\lambda}_{\br_e}$ is zero, then so is the coefficient of $\epsilon_j$ for any $j>i$; 
thus the induction hypothesis will apply.
%\textbf{Hadi:} What does this sentence mean? Also, the coefficient of $\epsilon_I$ in $\underline\lambda_{\br_e}$ is always zero. I suppose this is really an induction on the value of $I$, right?  After all, the value of $K_I$ is uniquely determined by $I$. }

First suppose that $J=m+1$ (which corresponds, for example, to the generic case).
%$\lambda$ is  $\br$-generic (which by convention is equivalent to $I>m$).
Then the sum over $(k,i)\succeq (K_J,J)$ is empty, so there is nothing to show.  

Next suppose the result for $J$.  Let $\lambda$ be such that the coefficient of $\epsilon_{J-1}$ in $\underline{\lambda}_{\br_e}$ is $0$.
 If $\ell_{J-1,\br}=\ell_{J-1,\br_e}$, then $J-1$ does not occur as the second index of any pair $(k,i)$ in \eqref{eq:inductionhyp}, so we have equality of left hand sides, and are done.
 Otherwise, let $k$ be the index such that $m-j_{2k-1}+1 \leq J-1 \leq m-j_{2k}$.  

If $J-1=m-j_{2k}$ then $k=K_{J-1}<K_J$.  Thus the left hand side of \eqref{eq:inductionhyp} is equal to the left hand side of \eqref{eq:lem1}.  By Lemma~\ref{lem:criticalcase}, with $k=K_{J-1}$, this is monoidally equivalent to
$$
\tau_{\br_e}(\underline{\lambda}_{\br_e}) + M_\br \sum_{k'>k} r_{odd,\br,k'} + (e_{m-j_{2k}}-e_{m+k}).
$$
Since $M_\br(\delta_{2k-1}-\epsilon_{m-j_{2k}}) = e_{m-j_{2k}}-e_{m+k}$, the latter formula is~\eqref{eq:inductionhyp} for $J-1$ and 
we are done by the induction hypothesis. 

Finally, if $m-j_{2k-1}+1 \leq J-1 < m-j_{2k}$ then $K_J=K_{J-1}=k$.  Thus the left hand side of  \eqref{eq:inductionhyp} is equal to the left hand side of \eqref{eq:lem2} with $i=J-1$.  By Lemma~\ref{lem:othercases}, this is equivalent (in the sense of $\sim_E$) to 
$$
\tau_{\br_e}(\underline{\lambda}_{\br_e}) + M_\br \sum_{(k',i')\succ(k,i)} (\delta_{2k'-1}-\epsilon_{i'})  + \left(\frac12e_{m-j_{2k}}+\frac12e_i-e_{m+k}\right).
$$
Since $M_\br(\delta_{2k-1}-\epsilon_{i}) = \frac12e_{m-j_{2k}}+\frac12e_i-e_{m+k}$, the latter formula is also~\eqref{eq:inductionhyp}
for $J-1$, and again
we are done.

Since the relation $\sim_E$ ensures that composing with $P_{\mu,\frac12}$, for any $\mu\in \mathscr{H}(m|n)$, gives the same value, this completes the proof of the theorem.
\end{proof}

\begin{remark}
Note that if $\br$ is very even, then 
$\tau_\br^{\mathrm{full}} = \tau_\br$ and thus Theorem~\ref{T:veryeven} 
is a special case of Theorem~\ref{T:fullcase}.  

However, when $\br$ is relatively even, then for each $k$ such that $j_{2k-1}-j_{2k}=1$, 
Definition~\ref{D:Cbfirst} sets $M(\delta_{2k-1}-\epsilon_{m-j_{2k}})= 0$ for each $M\in \mathcal{C}_\br$, whereas Definition~\ref{D:Cbfull} sets
$M(\delta_{2k-1}-\epsilon_{m-j_{2k}})=e_{m-j_{2k}}-e_{m+k}$  for each $M\in \mathcal{C}_\br^{\mathrm{full}}$.  Thus 
$\tau_\br^{\mathrm{full}}\neq \tau_\br$ but (as can be deduced from the proof of  Lemma~\ref{lem:criticalcase}), they are monoidally equivalent on all highest weights $\underline{\lambda}_{\br}$.  Interestingly, as Example~\ref{eg:surprisingexample-5} below shows, the choice of $M\delta_{2k-1}$ made in Definition~\ref{D:Cbfirst} cannot be extended to define a valid $\tau_\br$ in the non-relatively even case, and so Corollary~\ref{prop:reiscomp} is a separate solution to the refined Capelli eigenvalue problem that is neither a consequence nor a special case of Theorem~\ref{T:fullcase}.

Finally, note that the key step \eqref{eq:lem2} of Lemma~\ref{lem:othercases}, which is invoked precisely for the non-relatively even Borel subalgebras,  was not one of monoidal equivalence or supersymmetry, but rather of the weaker equivalence relation $\sim_E$, thus showing that monoidal equivalence alone does not suffice.  Nevertheless, a key result of \cite{MengyuanPhD} is that for any $\br'\in \mathcal{B}^d$, one can define \emph{piecewise} affine functions $\tau'$ that satisfy $\tau'(\underline{\lambda}_{\br'})\sim\tau_0(\underline{\lambda}_0)$ for all $\lambda\in \Omega_{m|n}$.
\end{remark}

It is natural to ask if the solution in Theorem~\ref{T:fullcase} is unique.   Here we give an example to illustrate this can be true.

\begin{eg}
\label{eg:surprisingexample-5}
Let $m=2$ and $n=1$, so that $\g=\gl(2|2)$.
The Weyl vector corresponding to the opposite standard Borel subalgebra is $\rho_0 = (1/2,3/2,-3/2,-1/2)$ and thus $X_0=M_0\rho_0=(-1/4,-3/4,1)$, where $M_0$ is given in~\eqref{E:M0}. 
Consider the Borel subalgebra $\br$ corresponding to the $\delta\epsilon$ sequence $\delta_2\epsilon_2\epsilon_1\delta_1$.  This is the smallest example with the property that $j_{2k-1,\br}-j_{2k,\br}\geq 2$ for some $1\leq k\leq n$.
Since $\br_0$ corresponds to $\delta_2\delta_1\epsilon_2\epsilon_1$, we know $\mathcal{G}_\br = \{ \delta_1-\epsilon_2, \delta_1-\epsilon_1\}$.  By Proposition~\ref{P:highestweightMN}, we compute the highest weights $\underline{\lambda}_\br$ for $\lambda \in \mathscr{H}(2|1)$ as in Table~\ref{Table:gl22}.
\begin{table}[htb]
\begin{tabular}{|c|c|c|c|}
\hline$\lambda$& $\underline{\lambda}_0$ & $\underline{\lambda}_\br$ & (condition) \\
\hline
$(r,s|t)$ & $-(2r, 2s, t, t)$ & $-(2r-1, 2s-1, t+2, t)$ &  $r\geq s > 0$,  $t\geq 0$, $r,s,t\in \Z$\\
$(r)$ & $-(2r, 0,0,0)$ & $-(2r-1, 0, 1, 0)$ & $r>0$, $r\in \Z$\\
$\emptyset$ & $(0,0,0,0)$&$(0,0,0,0)$&\\
\hline
\end{tabular}
\vspace{2mm}

\caption{The $\br_0$- and $\br$-highest weights of $W_\lambda$ for all  $\lambda\in\mathscr{H}(2|1)$.}
\label{Table:gl22}
\end{table}

Suppose now that  $\tau'(x)=M'x+X'$ is an affine map such that 
\begin{equation}
\label{eq:Pmu=}
P_{\mu,\frac12}(\tau'(\underline\lambda_\br))=
P_{\mu,\frac12}(\tau_0(\underline\lambda_0))
\quad\text{ for all }
\lambda\in \mathscr H(2|1).
\end{equation} 
%We have argued that there is no single affine map $\tau'$ satisfying $\tau'(\underline{\lambda}_\br)\sim \tau(\underline{\lambda}_0)$ for all $\lambda \in \mathscr{H}(2|1)$.  Moreover, the super Jack interpolation polynomials $P_{\mu,\frac12}$ form a basis for $\Lambda_{m,n,\frac12}$, and these are the polynomials which are constant on the orbits of $\sim$ on $\C^{m+n}$.    However, as shown in \cite{SVorbits}, there exist infinite orbits under $\sim$, and hence, families of orbits on a continuous parameter that cannot be separated by elements of $\Lambda_{m,n,\frac12}$. 
%
%One such family of orbits $\{\mathcal{O}(a)|a\in \C\}$ was computed explicitly in \cite[\S 5 (22)]{SVorbits}:  for each $a\in \C$ we have
%\begin{align*}
%\mathcal{O}(a) = \{&(\ell+a-\frac12,\ell+a,-2(\ell+a)+\frac12),
% (\ell+a+\frac12, \ell+a, -2(\ell+a)-\frac12), \\
% &(\ell+a,\ell+a-\frac12, -2(\ell+a)+\frac12), (\ell+a,\ell+a+\frac12, -2(\ell+a)-\frac12) \mid \ell \in \Z\}
%\end{align*}
%Evidently, any polynomial function constant on $\mathcal{O}(a)$ will be constant on $\mathcal{O}(b)$ for all $b\in \C$.    In fact, $\cup_{a\in \C} \mathcal{O}(a)$ is the union of the four affine lines:
%$$
%u_1-u_2=2u_1+v_1=\frac12, \quad u_2-u_1=2u_2+v_1=\pm\frac12.
%$$
% \textcolor{red}{On the other hand, the polynomials in $\Lambda_{m,n,1/2}$ do separate the finite orbits.}

Note that the orbits of $\underline{\lambda}_0$ under monoidal symmetry, for $\lambda$ $\br$-generic, are finite, so \eqref{eq:Pmu=} holds if and only if the elements are equivalent under monoidal symmetry, and as noted in Proposition~\ref{Prop:equalgeneric}, we must have $M' \in \mathcal{C}$, so that it takes the form
\begin{equation}\label{E:Mbgl22}
M' = \mat{-1/2 & 0 & a & -a\\ 0&-1/2 & b&-b\\ 0 & 0& -1/2 + c&-1/2-c}
\end{equation}
for some $a,b,c\in\C$.  Since in this case $r_\br = -\epsilon_1-\epsilon_2+2\delta_1$, the proposition also yields $X' - X_0 = M_\br r_\br = (\frac12+2a, \frac12+2b, -1+2c) \in \C^{2|1}$, so that
\begin{equation}\label{gl22diff}
X' = (1/4+2a, -1/4+2b, 2c),
\end{equation}
with $a,b,c$ as in \eqref{E:Mbgl22}.

Next consider the highest weights of the form $\underline{\lambda}_0=(2r,0,0,0)$ with $r>0$.  Then we have $M_0\underline{\lambda}_0+X_0 = (r-1/4, -3/4, 1)$.  Applying Lemma~\ref{lem:equforf}, and separate symmetry in the first two coordinates, we deduce that for $r \geq 1$ this vector is monoidally equivalent only to the four vectors
$$
(r-\frac14,-\frac34,1), \quad (r-\frac14, \frac14, 0), \quad (-\frac34, r-\frac14, 1), \quad \text{and} \quad (\frac14, r-\frac14, 1).
$$
Since this orbit is finite, it again coincides with its equivalence class under $\sim_E$.
Therefore to satisfy \eqref{eq:Pmu=} for each $r\geq 1$, the vector $M'\underline{\lambda}_\br+X'$ must occur on this list.  Applying \eqref{E:Mbgl22} and \eqref{gl22diff}, we compute
$$
M'\underline{\lambda}_\br+X'= M'(-2r+1, 0, -1, 0)+X' = (r-1/4+a, -1/4+b, 1/2+c),
$$
whence we must have $a=0$ and either $c=-b=1/2$ or $c=-b=-1/2$.  These two choices (which correspond to Definitions~\ref{D:Cbfirst} and \ref{D:Cbfull}, respectively), give 
$$
X' = (1/4, -5/4, 1) \quad \text{and} \quad X'=(1/4, 3/4, -1),
$$
respectively.

Finally, to satisfy \eqref{eq:Pmu=} when $\underline{\lambda}_\br=\underline{\lambda}_0=0$, we apply Lemma~\ref{lem:orbits} to deduce that we should have $X' \in \cup_{d\in \C}\mathcal{O}(d)$, which yields the unique choice $X' =  (\frac14, \frac34, -1)$.  

We conclude that $\tau'$ is uniquely determined by \eqref{eq:Pmu=} in this case, and is given by $X'=(\frac14, \frac34, -1)$ and 
$$
M' = \mat{-1/2 & 0 & 0 & 0\\ 0&-1/2 & 1/2&-1/2\\ 0 & 0& -1&0},
$$
which corresponds with the map $\tau_\br^{\mathrm{full}}$ of Definition~\ref{D:Cbfull}. 
\end{eg}

\color{black}

\appendix

\section{A normalization for Capelli operators}

\label{appxA}

Let $W$ be a $\mathbb Z_2$-graded vector space. Then there exists a natural pairing 
\[
\langle \cdot,\cdot\rangle:
 W^{\otimes d}\times (W^*)^{\otimes d}\to\C\ ,\ 
\langle v_1\otimes\cdots\otimes v_d,v_1^*\otimes \cdots\otimes v_d^*\rangle:=
\prod_{i=1}^d\langle v_{d-i+1},v_i^*\rangle,
\]
where $\langle v,v^*\rangle:=v^*(v)$. 
This gives an isomorphism
\begin{equation}
\label{eq:VdV*}
W^{\otimes d}\xrightarrow{\cong}
\left((W^*)^{\otimes d}\right)^*.
\end{equation}
Now recall the canonical isomorphisms
\begin{equation}
\label{eq:SrVdi}
\s^d(W)\cong (W^{\otimes d})^{S_d}
\quad\text{and}\quad
\p^d(W)\cong \s^d(W^*)\cong 
\left((W^*)^{\otimes d}\right)^{S_d}
.
\end{equation}
From the maps~\eqref{eq:VdV*} and~\eqref{eq:SrVdi} we obtain a linear map
\begin{equation}
\boldsymbol{\varphi}:
\s^d(W)\to \p^d(W)^*
\label{eq:frstmp}
\end{equation} 
as the composition of the maps
\begin{equation}
\label{eq:PdinSd}
\s^d(W)\xrightarrow{\boldsymbol{i}} W^{\otimes d}\xrightarrow{\cong } \left((W^*)^{\otimes d}\right)^*\xrightarrow{\boldsymbol{p}}\p^d(W)^*, 
\end{equation}
where $\boldsymbol{i}$ is induced by the inclusion  $(W^{\otimes d})^{S_d}\subseteq
W^{\otimes d}$ and $\boldsymbol{p}$ is induced by the projection that is dual to  the  inclusion $((W^*)^{\otimes d})^{S_d}\to (W^*)^{\otimes d}$. 
It is straightforward to verify that the map
$\boldsymbol{\varphi}$
is indeed an isomorphism of $\mathfrak{gl}(W)$-modules. In particular, given $v\in W_{\bar 0}$ and $v^*\in (W^*)_{\bar 0}$, for the $d$-tensors $v^{\otimes d}\in\s^d(W)$ and $(v^*)^{\otimes d}\in \p^d(W)$ we have
\[
\langle v^{\otimes d},(v^*)^{\otimes d}\rangle=\langle v,v^*\rangle^d.
\]
%Recall the canonical isomorphism $\mathscr D(W)\cong \s(W)$ where $\mathscr D(W)$ is the algebra of constant-coefficient differential operators on $W$. Using this isomorphism, 
We can also embed $\s(W)$ in $\mathrm{End}(\p(W))$ as follows. First for any homogeneous $v\in W\cong \s^1(W)$ we define a  superderivation $\partial_v$ on $\p(W)$ by setting
\[
\partial_v(v^*):=\langle v,v^*\rangle\quad\text{ for }v^*\in W^*\cong \p^1(W),
\]
and then extending $\partial_v$ to all of $\p(W)$ by the Leibniz rule 
\[
\partial_v(ab):=\partial_v(a)b+(-1)^{|v|\cdot |a|}a\partial_v(b). 
\]
The resulting derivations $\partial_v$ supercommute, and therefore we can extend the embedding $\partial:W\to\mathrm{End}(\p(W))$ to all of $\s(W)$ by the universal property of $\s(W)$, \emph{i.e.}, by setting 
\[
\partial_{v_1\cdots v_d}:=\partial_{v_1}\cdots \partial_{v_d}\text{ for }v_1,\ldots,v_d\in W.
\]
This map identifies $\s(W)$ with the algebra $\mathscr D(W)$ of constant-coefficient superdifferential operators on $\p(W)$.

Clearly $\partial_u(\p^d(W))\subseteq \p^{d-d'}(W)$ for $u\in \s^{d'}(W)$. In particular, when $d'=d$ we have $\p^0(W)\cong \C$, hence we obtain a linear map
\[
\boldsymbol{\varphi}':\s^d(W)\mapsto \p^d(W)^*\quad,\quad
u\mapsto\big[
p\mapsto \partial_up\big].
\] 
\begin{lem}
\label{lem:appx}
$\boldsymbol{\varphi}'=d!\boldsymbol{\varphi}$.
\end{lem}
\begin{proof}
The proof of Lemma~\ref{lem:appx} is an exercise in super linear algebra, and we omit the details. 
Here we only outline one proof strategy. First one proves that both  $\boldsymbol{\varphi}$ and $\boldsymbol{\varphi}'$ are $\mathfrak{gl}(W)$-module isomorphisms and $\s^d(W)$ is an irreducible $\mathfrak{gl}(W)$-module. Thus, it suffices to compare the values of $\boldsymbol{\varphi}$ and $\boldsymbol{\varphi}'$ on one nonzero vector. For instance, if $W_{\bar 0}\neq0$ then for nonzero $v\in W_{\bar 0}$ and $v^*\in W^*_{\bar 0}$ we have 
\[
\partial_v((v^*)^{\otimes d})=d\langle v,v^*\rangle (v^*)^{\otimes (d-1)},
\]
hence by induction we obtain
$
\partial_{v^{\otimes d}}((v^*)^{\otimes d})=
d!\langle v,v^*\rangle ^d$,
which completes the proof. 
\end{proof}
Lemma~\ref{lem:appx} can be used to fix a normalization for the Capelli operator 
$D^\lambda$ in Definition~\ref{defn:capelli_operator}
as follows. Recall the $\mathfrak g$-module decomposition 
\[
\p(W)\cong \bigoplus_{\lambda\in \Omega}W_\lambda.
\]
The map $\boldsymbol{\varphi}$ results in an isomorphism 
$
\s(W)\cong\bigoplus_{\lambda\in\Omega}W_\lambda^*$.
For $\lambda\in\Omega$, the  Capelli operator $D^\lambda$ is constructed as follows. We choose a basis $w_1^*,\ldots,w^*_{d_\lambda}$ for $W_\lambda$, and a dual basis $w_1,\ldots,w_{d_\lambda}$ for $W_\lambda^*$. From the definition of $D^\lambda$ it is straightforward to check that $D^\lambda=\sum_{i=1}^{d_\lambda} w_i^*\partial_{w_i}$.
In particular, if $W_\lambda\subseteq \p^d(W)$ then for $1\leq j\leq d_\lambda$ we have
\[
D^\lambda w_j^*=\sum_{i=1}^{d_\lambda} w_i^*\partial_{w_i}w_j^*
=d!
\sum_{i=1}^{d_\lambda}w_i^*\langle w_i,w_j^*\rangle =d! w_j^*.
\]
Consequently, $D^\lambda$ acts on $W_\lambda$ by the scalar $d!$.

%\textcolor{magenta}{I removed all references that were not cited; they are commented out in case you want to put any back in.}

\end{document}